\documentclass[a4paper, 12pt]{amsart}

\usepackage[utf8]{inputenc}
\usepackage[T1]{fontenc}

\usepackage{amsthm}
\usepackage{amssymb}
\usepackage{mathtools}
\usepackage{stmaryrd}	\SetSymbolFont{stmry}{bold}{U}{stmry}{m}{n}	% To avoid warnings
\usepackage{bm}
\usepackage{mathrsfs}
\usepackage{tikz-cd}
\usetikzlibrary{positioning, shapes.geometric, patterns, graphs, decorations.pathmorphing}

\usepackage{paralist}

\usepackage[colorlinks]{hyperref}
\hypersetup{
	linkcolor=blue,
	citecolor=teal!50!green,
	urlcolor=cyan,
	pdftitle = {Variations on themes of Sato},
	pdfauthor = {Wen-Wei Li}
}

\newcommand{\Z}{\ensuremath{\mathbb{Z}}}
\newcommand{\Q}{\ensuremath{\mathbb{Q}}}
\newcommand{\R}{\ensuremath{\mathbb{R}}}
\newcommand{\CC}{\ensuremath{\mathbb{C}}}
\newcommand{\A}{\ensuremath{\mathbb{A}}}
\newcommand{\F}{\ensuremath{\mathbb{F}}}

\newcommand{\Nrd}{\operatorname{Nrd}}	% Reduced norm
\newcommand{\Trd}{\operatorname{Trd}}	% Reduced trace
\newcommand{\topwedge}{\ensuremath{\bigwedge^{\mathrm{max}}}}
\newcommand{\Sym}{\operatorname{Sym}}

\newcommand{\twomatrix}[4]{ \ensuremath{\begin{pmatrix} #1 & #2 \\ #3 & #4 \end{pmatrix}} }	% Small 2 by 2 matrices

\newcommand{\Gal}{\operatorname{Gal}}

\newcommand{\dd}{\mathop{}\!\mathrm{d}}
\newcommand{\Schw}{\ensuremath{\mathcal{S}}}	% Schwartz space
\newcommand{\relgeq}[1]{\ensuremath{\underset{#1}{\geq}}}	% greater or equal relative to...
\newcommand{\relgg}[1]{\ensuremath{\underset{#1}{\gg}}}		% much greater relative to...

\newcommand{\lrangle}[1]{\ensuremath{\langle #1 \rangle}}

\newcommand{\Stab}{\ensuremath{\mathrm{Stab}}}

\newcommand{\identity}{\ensuremath{\mathrm{id}}}

\newcommand{\Hom}{\operatorname{Hom}}

\newcommand{\End}{\operatorname{End}}
\newcommand{\rightiso}{\ensuremath{\stackrel{\sim}{\rightarrow}}}

\newcommand{\Ker}{\operatorname{ker}}

\newcommand{\dotimes}[1]{\ensuremath{\underset{#1}{\otimes}}}

\newcommand{\Gm}{\ensuremath{\mathbb{G}_\mathrm{m}}}

\newcommand{\dtimes}[1]{\ensuremath{\underset{#1}{\times}}}

\newcommand{\GL}{\operatorname{GL}}
\newcommand{\SO}{\operatorname{SO}}

\newcommand{\SL}{\operatorname{SL}}
\newcommand{\Sp}{\operatorname{Sp}}
\newcommand{\Spin}{\operatorname{Spin}}
\theoremstyle{plain}
\newtheorem{proposition}{Proposition}
\newtheorem{lemma}[proposition]{Lemma}
\newtheorem{theorem}[proposition]{Theorem}
\newtheorem{corollary}[proposition]{Corollary}
\theoremstyle{definition}
\newtheorem{definition}[proposition]{Definition}
\newtheorem{definition-theorem}[proposition]{Definition-Theorem}
\newtheorem{definition-proposition}[proposition]{Definition-Proposition}
\newtheorem{remark}[proposition]{Remark}
\newtheorem{hypothesis}[proposition]{Hypothesis}

\newtheorem{example}[proposition]{Example}

\theoremstyle{definition}

\theoremstyle{plain}

\numberwithin{equation}{section}
\numberwithin{table}{section}
\numberwithin{proposition}{section}
\numberwithin{conj}{section}	% In the Introduction only

\renewcommand{\Re}{\operatorname{Re}}	% Symbol for the real and imaginary parts

\usepackage{amsmath}
\usepackage{geometry}
\usepackage{lmodern}

\title{Variations on themes of Sato}
\author[W.-W. Li]{Wen-Wei Li}
\address{School of Mathematical Sciences / Beijing International Center of Mathematical Research, Peking University \\
	No.\ 5 Yiheyuan Road, Beijing 100871, People's Republic of China.}
\thanks{This work is supported by NSFC-11922101.}
\email{wwli@bicmr.pku.edu.cn}

\date{}
\begin{document}

\begin{abstract}
	In the first part of this article, we review a formalism of local zeta integrals attached to spherical reductive prehomogeneous vector spaces, which partially extends M.\ Sato's theory by incorporating the generalized matrix coefficients of admissible representations. We summarize the basic properties of these integrals such as the convergence, meromorphic continuation and an abstract functional equation. In the second part, we prove a generalization that accommodates certain non-spherical spaces. As an application, the resulting theory applies to the prehomogeneous vector space underlying Bhargava's cubes, which is also considered by F.\ Sato and Suzuki--Wakatsuki in their study of toric periods. 
\end{abstract}

\maketitle

%{\scriptsize
%\begin{tabular}{ll}
%	\textbf{MSC (2010)} & 11S40; 11S90 43A85 \\
%	\textbf{Keywords} & Zeta integrals, prehomogeneous vector spaces, Capelli operators
%\end{tabular}}

\tableofcontents

\section{Introduction}\label{sec:intro}

\subsection{The works of Mikio Sato and Godement--Jacquet}
The origin of zeta integrals can be traced back to Tate's thesis, who exploited the integration of characters against Schwartz--Bruhat functions on the affine line to represent abelian $L$-functions, in both the local and global setting. In the 60's, Mikio Sato proposed a partial extension of Tate's construction by considering a \emph{prehomogeneous vector space} $(G, \rho, X)$; this means that $G$ is a linear algebraic group over a field $F$, and $\rho: G \to \GL(X)$ is a homomorphism of algebraic groups, such that $X$ has a Zariski-open orbit $X^+$. By convention, we let $\GL(X)$ act on the right of $X$. A rational function $f$ on $X$ is said to be a \emph{relative invariant} with eigencharacter $\omega: G \to \Gm$, if $f(x\rho(g)) = \omega(g)f(x)$ for all $x, g$.

Let us consider the case when $F$ is a local field of characteristic zero. For the sake of simplicity, assume temporarily that $X^+(F)$ is a single $G(F)$-orbit and that $\partial X := X \smallsetminus X^+$ is an irreducible divisor, defined by a relative invariant $f$. In such a local setting, Sato's zeta integral takes the form
\[ Z(\lambda, \xi_0) = \int_{X(F)} |f|^\lambda \xi_0 |\Omega| \]
where
\begin{compactitem}
	\item $\lambda \in \CC$,
	\item $\xi_0$ is a Schwartz--Bruhat function on $X(F)$,
	\item $\Omega \in \topwedge \check{X}$, $\Omega \neq 0$, so that $|\Omega|$ is a Haar measure on $X(F)$.
\end{compactitem}

The case $G = \Gm$ of acting on the affine line recovers Tate's thesis for the trivial character of $F^\times$. These integrals and their global avatars played a vital role in various questions of number theory. We refer to \cite{Ki03, Sa89} for a more complete survey of Sato's theory.

There are three basic yet non-trivial properties of local integrals $Z(\lambda, \xi)$.
\begin{description}
	\item[Convergence] The integrals converge when $\Re(\lambda) \gg 0$, and are holomorphic in $\lambda$ in the interior of that region.
	\item[Meromorphic continuation] They admit meromorphic continuation to all $\lambda \in \CC$, which are rational in $q^\lambda$ when $F$ is non-Archimedean with residual field $\F_q$.
	\item[Local functional equation] When $(G, \rho, X)$ is \emph{regular} and $\partial X$ is a hypersurface, its contragredient $(G, \check{\rho}, \check{X})$ is also prehomogeneous and the zeta integrals $Z(-\lambda, \mathcal{F}_\psi \xi_0)$ and $Z(\lambda, \xi_0)$ are equal up to a meromorphic/rational factor $\gamma(\lambda, \psi)$, where $\mathcal{F}_\psi$ is the Fourier transform relative to $\psi$. When $F$ is non-Archimedean, some conditions on the geometry of $\partial X$ are needed; see \cite{Sa89}.
\end{description}

We remark that when $X^+(F)$ comprises several $G(F)$-orbits, one should consider the integrals of $|f|^\lambda \xi$ on each orbit separately, and as a result, the functional equation will involve a $\gamma$-matrix instead of a scalar $\gamma$-factor. Also, when $\partial X$ comprises several codimension-one irreducible components, defined by the \emph{basic relative invariants} $f_1, \ldots, f_r$, the integration should involve $\prod_{i=1}^r |f_i|^{\lambda_i}$.

Another generalization of Tate's thesis is \emph{Godement--Jacquet theory} \cite{GJ72}. In the local case, it concerns a central simple $F$-algebra $D$ of dimension $n^2$ and the integral
\[ Z^{\mathrm{GJ}}\left( \lambda, v \otimes \check{v}, \xi_0 \right) := \int_{D^\times(F)} \lrangle{\check{v}, \pi(x)v} \left| \mathrm{Nrd}(x) \right|^{\lambda + \frac{n-1}{2}} \xi_0(x) \dd^\times x \]
where
\begin{compactitem}
	\item $\lambda \in \CC$,
	\item $\Nrd$ is the reduced norm of $D$,
	\item $\xi_0$ is a Schwartz--Bruhat function on $D(F)$,
	\item $v$ (resp.\ $\check{v}$) is a vector in the admissible representation $\pi$ of $D^\times(F)$ (resp.\ its contragredient),
	\item $\dd^\times x$ is a Haar measure on $D^\times(F)$.
\end{compactitem}

Identify $D$ with its dual by the reduced trace form. Again, we have three basic properties: convergence for $\Re(\lambda) \gg 0$, meromorphic/rational continuation, as well as a functional equation involving Fourier transform on $D(F)$ and a $\gamma$-factor $\gamma(\lambda, \pi, \psi)$ that is independent of $v, \check{v}$. The case $n=1$ recovers Tate's thesis.

Comparing these two formalisms, one may say that in Sato's framework, the basic properties of zeta integrals are deduced from the geometry of prehomogeneous vector spaces in a transparent manner, whereas Godement--Jacquet theory sometimes resorts to \textit{ad hoc} arguments to reduce to $n=1$. Also, there is a wider variety of choices of $(G, \rho, X)$, compared to the case of central simple algebras.

On the other hand, Godement--Jacquet theory affords the standard $L$-functions, whilst the $L$-functions arising from Sato's integrals are usually of a degenerate nature. This is unsurprising since Sato's integrals involve no admissible representation of $G(F)$ other than $|\omega|^\lambda$, where $\omega$ ranges over eigencharacters of relative invariants.

\subsection{Generalized prehomogeneous zeta integrals}
In \cite{LiLNM, Li18, Li19}, the author proposed a formalism which might be seen as a joint extension of Sato and Godement--Jacquet in the local case, by considering
\begin{itemize}
	\item a prehomogeneous vector space $(G, \rho, X)$ such that $G$ is connected reductive $F$-group, and $X^+$ is an affine spherical\footnote{Also known as \emph{absolutely spherical} by many other authors.} homogeneous $G$-space;
	\item when $F$ is non-Archimedean, we require that either $G$ is split and $X^+$ is wavefront, or $X^+$ is a symmetric space --- these conditions are inherited from \cite{SV17}, and can probably be improved;
	\item an admissible\footnote{For Archimedean $F$, this means an SAF, i.e.\ Casselman--Wallach representation.} representation $\pi$ of $G(F)$ with underlying vector space $V_\pi$;
	\item local zeta integrals of the form (Definition \ref{def:zeta})
	\[ Z_\lambda(\eta, v, \xi) = \int_{X^+(F)} \eta(v) |f|^\lambda \xi , \]
	where $\lambda \in \Lambda_{\CC}$, $\eta \in \mathcal{N}_\pi(X^+)$, $v \in V_\pi$ and $\xi \in \Schw(X)$.
\end{itemize}
Several explanations are in order.
\begin{itemize}
	\item We choose basic relative invariants $f_1, \ldots, f_r$ with eigencharacters $\omega_1, \ldots, \omega_r: G \to \Gm$, with $\Lambda := \bigoplus_{i=1}^r \Z \omega_i \subset \Hom(G, \Gm)$ and $\Lambda_{\CC} := \Lambda \otimes \CC$. Define $|f|^\lambda = \prod_{i=1}^r |f_i|^{\lambda_i}$ if $\lambda = \sum_{i=1}^r \omega_i \otimes \lambda_i$.
	\item Sphericity of $X^+$ means that it has an open Borel orbit, after base-change to the algebraic closure $\overline{F}$.
	\item We define $\mathcal{N}_\pi(X^+)$ to be the space of embeddings $\pi \to C^\infty(X^+)$ as smooth representations of $G(F)$, with the caveat that $C^\infty(X^+)$ is valued in the line bundle $\mathcal{L}^{1/2}$ of \emph{half-densities}, i.e.\ of the square-roots of volume forms.
	\item Likewise, $\Schw(X)$ is the space of Schwartz--Bruhat half-densities. Then $\eta(v)\xi |f|^\lambda$ is a $1$-density, i.e.\ measure-valued function, so its integration over $X^+(F)$ makes sense, if convergent.
\end{itemize}

The line bundle $\mathcal{L}^{1/2}$ is $G(F)$-equivariantly trivializable over $X^+(F)$ (Proposition \ref{prop:L-trivializable}). Hence one can switch to the scalar-valued picture, if desired.

Modulo some issues of shifts in $\lambda$ arising from half-densities, taking $\pi = \mathbf{1}$ recovers Sato's local zeta integrals for $(G, \rho, X)$, whilst taking the prehomogeneous vector space $(D^\times \times D^\times, \rho, D)$ with $x\rho(g_1, g_2) = g_2^{-1} x g_1$ recovers Godement--Jacquet theory. The sphericity implies that $\eta(v)$, the \emph{generalized matrix coefficients} of $\pi$ on $X^+$, are well-behaved; for example it can imply $\dim_{\CC} \mathcal{N}_\pi(X^+) < +\infty$ (Theorem \ref{prop:finite-mult}). We refer to \cite{SV17} for an overview of the harmonic analysis on spherical homogeneous spaces.

We remark that similar constructions have also been envisaged by Bopp--Rubenthaler \cite{BR05} and Fumihiro Sato \cite{Sa94,Sa06}, often with more restrictions on $(G, \rho, X)$ or $\pi$; it is also motivated by a proposal of Y.\ Sakellaridis \cite{Sak12}. In F.\ Sato's works, one can sometimes allow non-spherical $X^+$ by constraining $\pi$; we will return to this point later on.

A preliminary, yet unavoidable step is to establish the three basic properties for these integrals, namely:
\begin{description}
	\item[Convergence] $Z_\lambda(\eta, v, \xi)$ converges when $\Re(\lambda_i) \relgg{X} 0$ for all $i$, with $\lambda = \sum_{i=1}^r \omega_i \otimes \lambda_i$; the bound can be made uniform in $\eta$, $v$ and $\xi$.
	\item[Meromorphy/rationality] It admits a meromorphic continuation to all $\lambda \in \Lambda_{\CC}$, which is rational in $q^{\lambda_1}, \ldots, q^{\lambda_r}$ when $F$ is non-Archimedean with residual field $\F_q$.
	\item[Functional equation] The dual triplet $(G, \check{\rho}, \check{X})$ satisfies the same premises, and the corresponding $\check{Z}_\lambda$ satisfies
	\[ \check{Z}_\lambda\left( \check{\eta}, v, \mathcal{F}_\psi \xi \right) = Z_\lambda\left( \gamma(\lambda, \pi, \psi)(\check{\eta}), v, \xi \right) \]
	for all $\check{\eta} \in \mathcal{N}_\pi(\check{X}^+)$, $\xi \in \Schw(X)$, $v \in V_\pi$ and $\mathcal{F}_\psi$ is the Fourier transform of Schwartz--Bruhat half-densities. Here
	\[ \gamma(\pi, \lambda, \psi): \mathcal{N}_\pi(\check{X}^+) \to \mathcal{N}_\pi(X^+) \]
	is a uniquely determined meromorphic/rational family of $\CC$-linear maps. When $F$ is non-Archimedean, we also need a further condition on $\partial X$; see Hypothesis \ref{hyp:LFE}.
\end{description}

Thanks to half-densities, our formulation has fewer shifts or signs than the traditional setting; it also renders $\mathcal{F}_\psi$ equivariant. These results are obtained in \cite{LiLNM} for non-Archimedean $F$, in \cite{Li18, Li19} for Archimedean $F$, and the first half of this article is devoted to a survey of these results (Theorems \ref{prop:convergence-meromorphy}, \ref{prop:LFE}).

The reductive prehomogeneous vector spaces $(G, \rho, X)$ with $X^+$ spherical are also known as \emph{multiplicity-free spaces}. When $F = \overline{F}$ and $\mathrm{char}(F)=0$, the irreducible multiplicity-free spaces have been classified by V.\ Kac, and the general case is done in \cite{Le98}. Table \ref{tab:Kac} reproduces Kac's classification, following \cite[\S 11]{HU91} and \cite{Um98}, by recording only those with $X^+$ affine. By Matsushima's criterion \cite[Theorem 3.8]{Ti11}, $X^+$ is affine if and only if $H^\circ$, the identity connected components of generic stabilizers $H$ in $X$, are reductive.

\begin{table}[h]\centering
	\[\def\arraystretch{1.2}
	\begin{array}{c|c|c|c}
		G & X & H^\circ & \text{conditions} \\ \hline
		\GL(n) \times \GL(n) & \mathrm{M}_n & \GL(n) & \\
		\GL(n) & \Sym^n (F^n) & \SO(n) & \\
		\GL(2n) & \wedge^2 (F^{2n}) & \Sp(2n) & \\
		\mathrm{O}(n) \cdot \Gm & F^n & \SO(n-1) & \\
		\Sp(2n) \times \GL(2) & F^{2n} \otimes F^2 & \GL(2) & n \geq 2 \\
		\Sp(4) \times \GL(4) & F^4 \otimes F^4 & \Sp(4) & \\
		\Spin(7) \times \Gm & \mathrm{spin} & \mathrm{G}_2 & \\
		\Spin(9) \times \Gm & \mathrm{spin} & \Spin(7) & \\
		\mathrm{G}_2 \times \Gm & \dim=7 & \SL(3) & \\
		\mathrm{E}_6 \times \Gm & \dim=27 & \mathrm{F}_4 &
	\end{array}\]
	\caption{Irreducible multiplicity-free spaces with $X^+$ affine}
	\label{tab:Kac}
\end{table}

In Table \ref{tab:Kac}, the factors $\Gm$ act by dilation. For further explanations, see \cite{HU91, Um98}. The case of $\mathrm{E}_6 \times \Gm$ is also discussed in \cite[Example 3.12]{Li18}.

\subsection{Non-spherical cases}
The sphericity of $X^+$ is a rather strong requirement. It is possible to loosen it at the cost of constraining the admissible representation $\pi$ in question. This is motivated by the pioneering works of Fumihiro Sato \cite{Sa94, Sa06} and the recent work of Miyu Suzuki and Satoshi Wakatsuki \cite{SW20}. Both concentrate on the global picture related to the periods of automorphic forms.

To go beyond the spherical case, we consider a reductive prehomogeneous vector space $(G, \rho, X)$ with affine open orbit $X^+$, together with a normal subgroup $N \lhd G$, subject to the following conditions (Hypothesis \ref{hyp:beyond-sphericity}). Set $\overline{G} := G/N$.
\begin{itemize}
	\item $N$ acts freely on $X^+$;
	\item $G(F) \to \overline{G}(F)$ is surjective;
	\item $N$ is semisimple;
	\item the categorical quotient $Y^+ := X^+ /\!/ N$, as a homogeneous $\overline{G}$-space, satisfies the conditions of sphericity, etc.\ as before.
\end{itemize}
It turns out that $(G, \check{\rho}, \check{X})$ satisfies the same properties (Lemma \ref{prop:duality-beyond}). Since $\mathrm{char}(F) = 0$, it also implies that $X^+$ is an $N$-torsor.

The admissible representations $\pi$ of $G(F)$ are supposed to be trivial on $N(F)$, i.e.\ they are inflated from representations of $\overline{G}(F)$. Taking $N = \{1\}$ reduces to the original setting.

In this setting, one can define $\mathcal{N}_\pi(X^+)$ and $Z_\lambda(\eta, v, \xi)$ in exactly the same way. Note that $\dim_{\CC} \mathcal{N}_\pi(X^+) < +\infty$ by the same result of finiteness for $Y^+$.

In Theorems \ref{prop:LFE-beyond} and \ref{prop:convergence-meromorphy-beyond}, we will prove the three basic properties of zeta integrals in this setting: convergence, meromorphy/rationality, and the functional equation. Note that for the non-Archimedean functional equation, we impose the same Hypothesis \ref{hyp:LFE} on $\partial X$.

In particular, this includes the case considered by Suzuki--Wakatsuki \cite{SW20}. Note that the same prehomogeneous vector space also appeared in Bhargava's work \cite{Bha04}.

Let us conclude by a few words on the proof for non-spherical cases. The required properties are not a mere ``pull-back'' from $Y^+$ to $X^+$, since the Schwartz--Bruhat spaces and the Fourier transform live on the level of $X$ and $\check{X}$.

\begin{itemize}
	\item For the Archimedean case, say $F = \R$, the convergence and meromorphic continuation are established as in \cite{Li18, Li19}: we use the standard estimates as well as the holonomicity of $K$-finite generalized matrix coefficients on $Y^+(\R)$, which can be easily pulled back to $X^+(\R)$. The functional equation is slightly more delicate. We have to recast Knop's work \cite{Kn98} on invariant differential operators on multiplicity-free spaces into a suitable form. Specifically, we must calculate the top homogeneous component of the image of certain operators of Capelli-type (arising from relative invariants) under Knop's Harish-Chandra homomorphism, in terms of data on $X$. This is done in \S\S\ref{sec:invariant-theory}---\ref{sec:D}.
	
	\item For the non-Archimedean case, we follow the same arguments as in \cite{LiLNM}; in particular, we prove the rational continuation via Igusa theory. However, one has to consider toroidal embeddings for the non-spherical varieties $X^+$ and $X$, in a manner compatible with those of $Y^+$. To this end, we invoke the general theory by F.\ Knop and B.\ Krötz \cite{KK16}.
	
	On the other hand, the proof of non-Archimedean functional equation is proved in the same way as \cite[\S 6.3]{LiLNM}, under the same premises.
\end{itemize}

Since the proofs largely follow the same pattern as in the spherical case, we will only give brief sketches in \S\S\ref{sec:pf-arch}---\ref{sec:pf-nonarch}.

We consider only the local integrals in this article. Nonetheless, the long-term goal is to study their relation to the global integrals, and explore the arithmetic consequences.

\subsection*{Organization}
In \S\S\ref{sec:PVS}---\ref{sec:spherical-case}, we summarize the basic results on generalized prehomogeneous zeta integrals in the spherical case. In \S\ref{sec:beyond-spherical}, we present an extension to certain non-spherical cases, and state the main theorems. In \S\ref{sec:SW}, we illustrate the extended formalism in the setting of Suzuki--Wakatsuki. The proofs for the non-spherical case occupy \S\S\ref{sec:invariant-theory}---\ref{sec:pf-nonarch}.

\subsection*{Acknowledgements}
The results in \S\S\ref{sec:PVS}---\ref{sec:spherical-case} were presented during the First JNT Biennial Conference, held in Cetraro, July 2019. The author is deeply grateful to the organizing committee for providing him this opportunity. Thanks also go to Miyu Suzuki and Satoshi Wakatsuki, for kindly sharing their preprint \cite{SW20} and urging the author to think about non-spherical cases. This work is supported by NSFC-11922101.

\subsection*{Conventions}
The normalized absolute value on a local field is denoted by $|\cdot|$.

For any variety $X$ over a field $F$ and an extension $E|F$ of fields, we write $X(E)$ for the set of $E$-points of $X$. We also write $X \dtimes{F} E$ for its base-change to $E$. When $X$ is smooth, the tangent (resp.\ cotangent) bundle is denoted by $TX$ (resp.\ $T^* X$). The algebra of regular functions on $X$ is denoted by $F[X]$. The function field of an irreducible variety $X$ is denoted by $F(X)$.

The Lie algebra of $G$ is denoted by $\mathfrak{g}$, and so forth. By convention, algebraic groups act on the right of varieties. If $G$ is a connected reductive group, a $G$-variety means a normal irreducible variety with (right) $G$-action; if the action is transitive, it is called a homogeneous $G$-space. In particular, for a finite-dimensional vector space $X$, the algebraic group $\GL(X)$ acts on the right of $X$.

In contrast, the representations of locally compact groups act on the left. The underlying space of such a representation $\pi$ is written as $V_\pi$. When a group $G$ acts on the right of a space (resp.\ variety) $X$, it acts on the left of function on $X$ (resp.\ regular functions on $X$) by $(gf)(x) = f(xg)$.

The dual of a vector space $X$ is denoted by $\check{X}$. The top exterior power of $X$ is denoted by $\topwedge X$ when $\dim X$ is finite. The contragredient of a representation $\rho$ is denoted by $\check{\rho}$. The trivial representation is denoted by $\mathbf{1}$.

The space of $n \times n$ matrices is denoted by $\mathrm{M}_n$. For a linear algebraic group $G$ over a field $F$, we write $\mathbf{X}^*(G) := \Hom_{F\text{-group}}\left(G, \Gm \right)$, which is an additive group.

The discriminant of a quadratic form $q$ is denoted by $\mathrm{disc}(q)$.

\section{Reductive prehomogeneous vector spaces}\label{sec:PVS}

Let $F$ be a field of characteristic zero with algebraic closure $\overline{F}$. Let $G$ be a connected reductive $F$-group.

\begin{definition}
	Let $Z$ be a homogeneous $G$-space. It is said to be \emph{spherical} (also known as absolutely spherical) if there exists an open $B$-orbit in $Z \dtimes{F} \overline{F}$, where $B \subset G \dtimes{F} \overline{F}$ is any Borel subgroup.
\end{definition}

Suppose that $G$ acts on the right of a finite-dimensional $F$-vector space $X$ through an algebraic homomorphism $\rho: G \to \GL(X)$. We say the triplet $(G, \rho, X)$ is a \emph{reductive prehomogeneous vector space} if there is a Zariski-open dense $G$-orbit in $X$, hereafter denoted as $X^+$. We also write $\partial X := X \smallsetminus X^+$.

The proofs of the facts below can be found in \cite{Ki03} when $F = \overline{F}$. The general case follows by Galois descent and an application of Hilbert's Theorem 90; see the discussions in \cite[\S 2.1]{Li19}.

A nonzero rational function $f \in F(X)$ is called a \emph{relative invariant} if there exists $\omega \in \mathbf{X}^*(G)$ such that $f(xg) = \omega(g) f(x)$ for all $(x, g) \in X \times G$. The unique character $\omega$ here is called the \emph{eigencharacter} associated with $f$; note that $\omega$ determines $f$ up to $F^\times$. Relative invariants are automatically homogeneous. By varying $f$, the eigencharacters form a subgroup $\mathbf{X}^*_\rho(G)$ of $\mathbf{X}^*(G)$. We have
\[ \mathbf{X}^*_\rho(G) = \mathbf{X}^*_{\rho \dotimes{F} \overline{F}}\left( G \dtimes{F} \overline{F} \right)^{\Gal(\overline{F}|F)}. \]

The general theory of prehomogeneous vector spaces affords us a set of eigencharacters $\omega_1, \ldots, \omega_r$ such that $\mathbf{X}^*_\rho(G) = \bigoplus_{i=1}^r \Z\omega_i$. To each $\omega_i$ is associated a relative invariant $f_i \in F[X]$, unique up to $F^\times$. We call $f_1, \ldots, f_r$ the \emph{basic relative invariants} of $(G, \rho, X)$: they define the codimension-one irreducible components of $\partial X$. In particular, $\prod_{i=1}^r f_i^{a_i} \in F[X]$ if and only if $a_i \geq 0$ for all $i$. Accordingly, we call $\omega_1, \ldots, \omega_r$ the \emph{basic eigencharacters} of $(G, \rho, X)$.

If $f \in F(X)$ is a relative invariant, then $f^{-1} \dd f$ defines a $G$-equivariant morphism $X^+ \to \check{X}$ between $F$-varieties. We say that $f$ is \emph{non-degenerate} if $f^{-1} \dd f$ is dominant. We say that $(G, \rho, X)$ is \emph{regular} if it admits a non-degenerate relative invariant.

Note that $\partial X$ is a hypersurface if and only if $X^+$ is affine (see \cite[Theorem 2.28]{Ki03}). To $(G, \rho, X)$ is associated the dual triplet $(G, \check{\rho}, \check{X})$, where $\check{\rho}$ is the contragredient of $\rho$.

\begin{proposition}\label{prop:PVS-generality}
	Let $(G, \rho, X)$ be a reductive prehomogeneous vector space. Assume that $X^+$ is affine. Then the following holds:
	\begin{compactitem}
		\item $(G, \rho, X)$ is regular: in fact, $(\det\rho)^2$ is the eigencharacter of some non-degenerate relative invariant;
		\item $(G, \check{\rho}, \check{X})$ is a regular prehomogeneous vector space as well;
		\item every non-degenerate relative invariant $f \in F(X)$ induces a $G$-equivariant isomorphism $f^{-1} \dd f: X^+ \rightiso \check{X}^+$;
		\item $\mathbf{X}^*_{\check{\rho}}(G) = \mathbf{X}^*_\rho(G)$ and $\omega_1^{-1}, \ldots, \omega_r^{-1}$ are the basic eigencharacters for $(G, \check{\rho}, \check{X})$;
		\item we may choose a non-degenerate relative invariant $f \in F[X]$ (resp.\ $\check{f} \in F[\check{X}]$) whose zero locus is $\partial X$ (resp.\ $\partial \check{X}$), such that $f$ and $\check{f}$ have opposite eigencharacters.
	\end{compactitem}
\end{proposition}
\begin{proof}
	These properties are proved in \cite[\S 2.1]{Li19}, under the tacit assumptions that $X^+$ is spherical and $F = \R$. The arguments therein carry over to the general case. For instance, regularity follows immediately from \cite[Proposition 2.24]{Ki03} as $X^+$ is affine.
\end{proof}

\begin{remark}
	For non-reductive $G$, one can still define the prehomogeneous vector spaces $(G, \rho, X)$. The notions of relative invariants and non-degeneracy also carry over to the general case. See \cite{Ki03}.
\end{remark}

Hereafter, we assume that $F$ is a local field. To each $F$-analytic manifold $Y$ and $s \in \R$, we may define the line bundle $\mathcal{L}^s$ of \emph{$s$-densities}: an $s$-density can be thought as an $s$-th power of a volume form on $Y$. The integration $\int_Y \eta$ of a $1$-density $\eta$ on $Y$ makes sense, provides that it converges. An $s$-density and a $t$-density can be naturally multiplied to yield an $(s+t)$-density. The $\frac{1}{2}$-densities are also called \emph{half-densities}. Therefore, one can talk about square-integrable half-densities, which form the Hilbert space $L^2(Y)$.

If $G(F)$ acts on $Y$, the density bundles carry natural $G(F)$-equivariant structures. Specializing to the situation of Proposition \ref{prop:PVS-generality}, every $\Omega \in \topwedge \check{X}$ gives rise to the translation-invariant $s$-density $|\Omega|^s$; its restriction to $X^+(F)$ is still an $s$-density. If $\Omega \neq 0$, then $|\Omega|$ gives a Haar measure on $X(F)$.

\begin{proposition}\label{prop:L-trivializable}
	Under the assumptions of Proposition \ref{prop:PVS-generality}, the line bundle $\mathcal{L}^s$ on $X(F)$ is equivariantly trivializable for any $s \in \R$. Specifically, let $\phi \in F(X)$ be a relative invariant with eigencharacter $(\det\rho)^2$ and $\Omega \in \topwedge \check{X} \smallsetminus \{0\}$, then $|\phi|^{-s/2} |\Omega|^s$ is a $G(F)$-invariant and non- vanishing section of $\mathcal{L}^s$ over $X^+(F)$.
\end{proposition}
\begin{proof}
	See \cite[Lemma 6.6.1]{LiLNM}.
\end{proof}

\begin{definition}
	For any finite-dimensional $F$-vector space $X$, we denote by $\Schw_0(X)$ the space of scalar-valued Schwartz--Bruhat functions on $X(F)$, and by $\Schw(X) = \Schw_0(X) |\Omega|^{1/2}$ the space of Schwartz--Bruhat half-densities on $X(F)$. Here $\Omega \in \topwedge \check{X} \smallsetminus \{0\}$ is arbitrary.
\end{definition}

Given a representation $\rho: G \to \GL(X)$, we deduce left $G(F)$-actions on $\Schw_0(X)$ and on $\Schw(X)$. Note that $G(F)$ dilates $|\Omega|^{1/2}$.

Fix an additive character $\psi$ of $F$ and let $\lrangle{\cdot, \cdot}: \check{X} \times X \to F$ be the canonical pairing, which induces a pairing between $\topwedge \check{X}$ and $\topwedge X$. Given $\Omega$ as above, the usual Fourier transform is
\[\begin{tikzcd}[row sep=tiny]
	\mathcal{F}_{\psi, |\Omega|}: \Schw_0(X) \arrow[r] & \Schw_0(\check{X}) \\
	\xi_0 \arrow[mapsto, r] & {\left[ \check{x} \mapsto \displaystyle\int_{x \in X(F)} \xi_0(x) \psi(\lrangle{\check{x}, x}) |\Omega| \right]}.
\end{tikzcd}\]

The following easy fact explains the usefulness of half-densities. See \cite[\S 6.1]{LiLNM} or \cite[\S 2.3]{Li19} (where we assumed $F = \R$) for further explanations.

\begin{definition-proposition}
	Let $X$ be a finite-dimensional $F$-vector space. Choose any $\Omega \in \topwedge \check{X} \smallsetminus \{0\}$ and take the unique $\Psi \in \topwedge X$ with $\lrangle{\Phi, \Psi}=1$. Define the Fourier transform of half-densities as
	\[\begin{tikzcd}[row sep=tiny]
		\mathcal{F}_\psi: \Schw(X) \arrow[r] & \Schw(\check{X}) \\
		\xi = \xi_0 |\Omega|^{1/2} \arrow[mapsto, r] & \mathcal{F}_{\psi, |\Omega|}(\xi_0) |\Psi|^{1/2}.
	\end{tikzcd}\]
	This is independent of the choice of $\Omega$ and yields a $G(F)$-equivariant isomorphism $\Schw(X) \rightiso \Schw(\check{X})$. It extends to $L^2(X) \rightiso L^2(\check{X})$.
\end{definition-proposition}

\begin{remark}
	There is a slightly different, ``self-dual'' normalization $\mathcal{F}_\psi^{\mathrm{sd}} := A(\psi)^{-1/2} \mathcal{F}_\psi$ in \cite[Remark 2.11]{Li19}; it satisfies $\mathcal{F}_{-\psi}^{\mathrm{sd}} \mathcal{F}_\psi^{\mathrm{sd}} = \identity_{\Schw(X)}$ and extends to an isometry $L^2(X) \rightiso L^2(\check{X})$. We refer to \textit{loc.\ cit.} for details.
\end{remark}

We will also need the following decomposition. Consider a triplet $(G, \rho, X)$ with $X^+$ affine as before, over $F = \R$. Choose the basic relative invariants $f_1, \ldots, f_r \in \R[X]$. Let $A_G \subset G$ be the maximal split central torus, and let $A_G(\R)^\circ$ be the identity connected component of $A_G(\R)$. Let $H_G: G(\R) \to \mathfrak{a}_G := \Hom(\mathbf{X}^*(G), \R)$ be the Harish-Chandra homomorphism. Set $G(\R)^1 := \Ker(H_G)$. It is well-known that $H_G: A_G(\R)^\circ \rightiso \mathfrak{a}_G$, and multiplication yields an isomorphism of real Lie groups
\[ A_G(\R)^\circ \times G(\R)^1 \rightiso G(\R). \]
We define
\begin{align*}
	G(\R)_\rho & := \left\{ g \in G(\R) : \forall \chi \in \mathbf{X}^*_\rho(G), \; |\chi(g)| = 1 \right\}, \\
	X^+(\R)_\rho & := \left\{ x \in X^+(\R) : \forall 1 \leq i \leq r, \; |f_i(x)|=1 \right\}.
\end{align*}
Then $G(\R)_\rho \supset G(\R)^1$ and $G(\R)_\rho$ acts on the right of $X^+(\R)_\rho$.

Observe that $\mathfrak{a}_G \twoheadrightarrow \mathfrak{a}_\rho := \Hom(\mathbf{X}^*_\rho(G), \R)$. Choose a splitting to realize $\mathfrak{a}_\rho$ as a direct summand of $\mathfrak{a}_G$, and set $A_\rho := H_G^{-1}(\mathfrak{a}_\rho) \subset A_G(\R)^\circ$. Therefore $(|\omega_1|, \ldots, |\omega_r|)$ induces $A_\rho \rightiso (\R^\times_{>0})^r$.

Finally, define $r: X^+(\R) \to A_\rho$ as the map characterized by $|f_i(y)| = |\omega_i(r(y))|$ for $i = 1, \ldots, r$.

\begin{proposition}[{\cite[Proposition 6.1]{Li19}}]\label{prop:rho-decomposition}
	There are isomorphisms of real analytic manifolds
	\[\begin{tikzcd}[row sep=tiny]
		G(\R)_\rho \times A_\rho \arrow[r, "\sim"] & G(\R) \\
		(g, a) \arrow[mapsto, r] & ga
	\end{tikzcd}\]
	and
	\[\begin{tikzcd}[row sep=tiny]
		X^+(\R)_\rho \times A_\rho \arrow[r, "\sim"] & X^+(\R) \\
		(x, a) \arrow[mapsto, r] & xa \\
		\left( y r(y)^{-1}, r(y) \right) & y \arrow[mapsto, l] .
	\end{tikzcd}\]
\end{proposition}
\begin{proof}
	The proof is identical to that in \textit{loc.\ cit.}, since it uses only the basic properties of reductive prehomogeneous vector spaces; in particular, $X^+$ does not have to be spherical.
\end{proof}

\section{Generalized zeta integrals}\label{sec:gen-zeta}
In this section, $F$ is a local field of characteristic zero. Let $G$ be a connected reductive $F$-group as before, we will consider the representations of the locally compact group $G(F)$ on $\CC$-vector spaces, assumed to be continuous in the Archimedean case. For representations $\pi_1$ and $\pi_2$ of $G(F)$, the space $\Hom_{G(F)}(\pi_1, \pi_2)$ consists of $G(F)$-equivariant linear maps from $V_{\pi_1}$ to $V_{\pi_2}$, assumed to be continuous when $F$ is Archimedean.

By an \emph{admissible representation} of $G(F)$, we shall mean:
\begin{itemize}
	\item a smooth admissible representation of $G(F)$ of finite length, when $F$ is non-Archimedean;
	\item an SAF representations of $G(F)$ (smooth, admissible of moderate growth, Fréchet --- see \cite{BK14}) of finite length, when $F$ is Archimedean.
\end{itemize}

Let $X^+$ be a homogeneous $G$-space. Following \cite[\S 4.1]{LiLNM}, we make the
\begin{definition}
	Let $C^\infty(X^+) = C^\infty\left( X^+ ; \mathcal{L}^{1/2} \right)$ be the space of $C^\infty$-half-densities on $X^+(F)$, which is a smooth representation under the obvious left $G(F)$-action. More precisely, it is the smooth $G(F)$-representation associated to $C\left( X^+ ; \mathcal{L}^{1/2}\right)$ by taking smooth vectors. For Archimedean $F$ it is a smooth Fréchet representation.
	
	For every admissible representation $\pi$ of $G(F)$, we define
	\[ \mathcal{N}_\pi(X^+) := \Hom_{G(F)}\left(\pi, C^\infty(X^+) \right). \]
	For every $\eta \in \mathcal{N}_\pi(X^+)$ and $v \in V_\pi$, the $\mathcal{L}^{1/2}$-valued function $\eta(v)$ is called a \emph{generalized matrix coefficient} of $\pi$.
\end{definition}

Note that when $\mathcal{L}^{1/2}$ is equivariantly trivializable on $X^+(F)$, such as the case Proposition \ref{prop:L-trivializable}, $C^\infty(X^+)$ is isomorphic to the usual scalar-valued $C^\infty$ space on $X^+(F)$. In that case, $\mathcal{N}_\pi(X^+)$ is the familiar object studied in \emph{relative harmonic analysis}, at least for irreducible $\pi$. Cf.\ Remark \ref{rem:eta-0}.

Following \cite{LiLNM, Li19}, the following conditions on $X^+$ will be imposed in \S\ref{sec:spherical-case}.

\begin{hypothesis}\label{hyp:sphericity}
	We assume that either
	\begin{compactitem}
		\item $F$ is Archimedean and $X^+$ is a spherical homogeneous $G$-space,
		\item $F$ is non-Archimedean, $G$ is split and $X^+$ is a wavefront spherical homogeneous $G$-space in the sense of \cite[p.23]{SV17},
		\item or $F$ is non-Archimedean and $X^+$ is a symmetric space (hence wavefront spherical) under $G$.
	\end{compactitem}
\end{hypothesis}

The conditions in the non-Archimedean case are inherited from \cite{SV17}.

Hereafter, $(G, \rho, X)$ will be a reductive prehomogeneous vector space such that $X^+$ is affine. By Proposition \ref{prop:PVS-generality}, $(G, \check{\rho}, \check{X})$ satisfies the same requirements. Denote by $\omega_1, \ldots, \omega_r$ the basic eigencharacters for $(G, \rho, X)$. Let us state the finiteness of multiplicities as follows.

\begin{theorem}\label{prop:finite-mult}
	Under the Hypothesis \ref{hyp:sphericity}, we have $\dim_{\CC} \mathcal{N}_\pi(X^+) < +\infty$ for every admissible representation $\pi$ of $G(F)$.
\end{theorem}

This crucial fact is proved by various authors. For complete references, we refer to \cite[Theorem 3.2]{Li19} or \cite{KO13} for the Archimedean case, and to \cite[Theorem 5.1.5]{SV17} for the non-Archimedean case.

Given $(G, \rho, X)$, we set $\Lambda_A := \mathbf{X}^*_\rho(G) \dotimes{\Z} A$ for any commutative ring $A$. For $\lambda = \sum_{i=1}^r \omega_i \otimes \lambda_i \in \Lambda_{\R}$, we write $\lambda \relgg{X} 0$ if $\lambda_i \gg 0$ for each $1 \leq i \leq r$. Similarly for $\lambda \relgeq{X} 0$, etc. For $\lambda = \sum_{i=1}^r \omega_i \otimes \lambda_i \in \Lambda_{\CC}$, we write
\begin{equation}\label{eqn:omega-lambda}\begin{aligned}
	|\omega|^\lambda & := \prod_{i=1}^r |\omega_i|^{\lambda_i}: G(F) \to \CC^\times, \\
	|f|^\lambda & := \prod_{i=1}^r |f_i|^{\lambda_i}: X(F) \to \CC,
\end{aligned}\end{equation}
where $f_1, \ldots, f_r$ are chosen basic relative invariants, so that $|f|^\lambda$ has eigencharacter $|\omega|^\lambda$ under $G(F)$-action.

\begin{definition}[Generalized zeta integrals]\label{def:zeta}
	For $(G, \rho, X)$ as above, let $\pi$ be an admissible representation of $G(F)$ and let $\eta \in \mathcal{N}_\pi(X^+)$. For all $v \in V_\pi$, $\xi \in \Schw(X)$ and $\lambda \in \Lambda_{\CC}$ with $\Re(\lambda) \relgg{X} 0$, we set
	\[ Z_\lambda\left(\eta, v, \xi \right) := \int_{X^+(F)} \eta(v) |f|^\lambda \xi, \]
	granting the convergence of this integral.
\end{definition}

As $\eta(v) |f|^\lambda \xi$ is a $1$-density on $X^+(F)$, one can talk about its integral. The issue of convergence will be discussed in \S\ref{sec:spherical-case} (under Hypothesis \ref{hyp:sphericity}) and \S\ref{sec:beyond-spherical}. Observe that the convergence is trivial when $\xi$ is compactly supported on $X^+(F)$.

\begin{remark}
	Consider the group of characters $\mathcal{T} := \left\{ |\omega|^\lambda: \lambda \in \Lambda_{\CC} \right\}$ of $G(F)$. When $F$ is Archimedean, $\mathcal{T}$ is isomorphic to $\Lambda_{\CC}$ by $|\omega|^\lambda \mapsto \lambda$. When $F$ is non-Archimedean with residue field $\F_q$, it is naturally isomorphic to $\left( \CC^\times \right)^r$, by mapping $|\omega|^{\sum_i \omega_i \otimes \lambda_i}$ to $(q^{\lambda_1,}, \ldots, q^{\lambda_r})$.

	Therefore, assuming $\mathcal{N}_\pi(X^+)$ is finite-dimensional, it makes sense to talk about meromorphic (resp.\ rational) families in $\mathcal{N}_\pi(X^+)$ indexed by $\mathcal{T}$, when $F$ is Archimedean (resp.\ non-Archimedean).
	
	We will talk about the meromorphy or rationality of zeta integrals $Z_\lambda(\eta, v, \xi)$ in this sense, when $\lambda \in \Lambda_{\CC}$ (or rather $|\omega|^\lambda \in \mathcal{T}$) varies. This is clearly unaffected by the choice of basic relative invariants $f_1, \ldots, f_r$.
\end{remark}

\begin{remark}\label{rem:eta-0}
	One can get rid of half-densities by taking $\phi$ and $\Omega$ as in Proposition \ref{prop:L-trivializable}. Then every $\eta \in \mathcal{N}_\pi(X^+)$ can be written as
	\[ \eta = \eta_0 |\phi|^{-1/4} |\Omega| \]
	so that $\eta_0 \in \Hom_{G(F)}(\pi, C^\infty(X^+(F); \CC))$, i.e.\ it yields scalar-valued generalized matrix coefficients. Let us rescale $\phi$ to ensure $|\phi|^{1/4} = |f|^{\lambda_0}$ with $\lambda_0 \in \frac{1}{4} \Lambda_{\Z}$. Writing $\xi \in \Schw(X)$ as $\xi = \xi_0 |\Omega|^{1/2}$, we have
	\[ Z_\lambda(\eta, v, \xi) = \int_{X^+(F)} \eta_0(v) |f|^{\lambda - \lambda_0} \xi_0 |\Omega|, \quad \xi_0 \in \Schw_0(X). \]
\end{remark}

\section{Basic properties in the spherical case}\label{sec:spherical-case}

Let $F$ be a local field of characteristic, $G$ be a connected reductive $F$-group, and $(G, \rho, X)$ be a reductive prehomogeneous vector space such that $X^+$ is affine. Choose the basic relative invariants $f_1, \ldots, f_r \in F[X]$ with eigencharacters $\omega_1, \ldots, \omega_r \in \mathbf{X}^*_\rho(X)$.

The following statements are simplified forms of the main results from \cite[Chapter 6]{LiLNM} (non-Archimedean case) and \cite{Li19} (Archimedean case, reducing to $F = \R$); we omit the issues about the separate continuity of $Z_\lambda(\eta, \cdot, \cdot)$, the ``denominator'' of $Z_\lambda$ for Archimedean $F$, and the properties of $\gamma$-factors. The proofs thereof will be reviewed when we extend them to certain non-spherical cases.

\begin{theorem}[see {\cite[Theorem 6.2.7]{LiLNM}} and {\cite[Theorems 3.10, 3.12]{Li19}}]\label{prop:convergence-meromorphy}
	Let $\pi$ be an admissible representation of $G(F)$. Under the Hypothesis \ref{hyp:sphericity}, there exists $\kappa = \kappa(\pi) \in \Lambda_{\R}$, depending solely on $(G, \rho, X)$ and $\pi$, such that $Z_\lambda(\eta, v, \xi)$ is defined by a convergent integral for all $\eta, v, \xi$ whenever $\Re(\lambda) \relgeq{X} \kappa$, and $Z_\lambda(\eta, v, \xi)$ is holomorphic in $\lambda$ in the interior of that region.
	
	In this case, the function $\lambda \mapsto Z_\lambda(\eta, v, \xi)$ extends to a meromorphic family indexed by $\mathcal{T} = \left\{ |\omega|^\lambda : \lambda \in \Lambda_{\CC} \right\}$, for each $(\eta, v, \xi)$. It is a rational family when $F$ is non-Archimedean.
\end{theorem}

In order to state the functional equation, we fix an additive character $\psi$ for $F$ to define the Fourier transform $\mathcal{F}_\psi: \Schw(X) \to \Schw(\check{X})$ for half-densities. We also choose the basic relative invariants $\check{f}_1, \ldots, \check{f}_r \in F[\check{X}]$ with eigencharacters $\omega_1^{-1}, \ldots, \omega_r^{-1}$. The corresponding zeta integral is denoted as $\check{Z}_\lambda(\cdots)$. In this case, $X^+ \simeq \check{X}^+$ as homogeneous $G$-spaces, and it makes sense to talk about meromorphic or rational families of linear maps $\mathcal{N}_\pi(\check{X}^+) \to \mathcal{N}_\pi(X^+)$ indexed by $\mathcal{T}$, when $\dim_{\CC} \mathcal{N}_\pi(X^+) < + \infty$.

We need extra assumptions to prove the functional equation in the non-Archimedean case. The following is taken from \cite[Hypothesis 6.3.2]{LiLNM}. Note that even when $\pi = \mathbf{1}$, i.e.\ for Sato's prehomogeneous local zeta integrals, a similar condition has been imposed in \cite[p.474 (A.2)]{Sa89} for their functional equations.

\begin{hypothesis}\label{hyp:LFE}
	When $F$ is non-Archimedean, we assume that for every $y \in (\partial X)(F)$ with stabilizer $H := \Stab_G(y)$, there exists a parabolic subgroup $P \subset G$, with Levi quotient $M := P/U_P$, such that
	\begin{compactitem}
		\item $U_P \subset H \subset P$, so that we can set $H_M := H/U_P$;
		\item the restriction of $\mathcal{T}$ to $(Z_M \cap H_M)(F)$ contains a complex torus of dimension $> 0$ (note that the characters of $\mathcal{T}$ are trivial on $U_P(F)$).
	\end{compactitem}
\end{hypothesis}

\begin{theorem}[see {\cite[Theorem 6.3.6]{LiLNM}} and {\cite[Theorem 3.13]{Li19}}]\label{prop:LFE}
	Let $\pi$ be an irreducible admissible representation of $G(F)$. Under the Hypotheses \ref{hyp:sphericity} and \ref{hyp:LFE}, there exists a unique meromorphic family of linear maps
	\[ \gamma(\pi, \lambda, \psi): \mathcal{N}_\pi(\check{X}^+) \to \mathcal{N}_\pi(X^+), \quad \lambda \in \Lambda_{\CC} \]
	indexed by $\mathcal{T}$ and rational in the non-Archimedean case, such that
	\[ \check{Z}_\lambda\left( \check{\eta}, v, \mathcal{F}_\psi \xi \right) = Z_\lambda\left( \gamma(\lambda, \pi, \psi)(\check{\eta}), v, \xi \right) \] 
	for all $\check{\eta} \in \mathcal{N}_\pi(\check{X}^+)$, $v \in V_\pi$, $\xi \in \Schw(X)$, as meromorphic or rational families.
\end{theorem}

We conclude this section with two basic instances of this framework. We choose $\Omega \in \topwedge \check{X}$ and $\Psi \in \topwedge X$ such that $\lrangle{\Omega, \Psi} = 1$.

\begin{example}[Sato--Shintani]
	Take $(G, \rho, X)$ subject to the Hypothesis \ref{hyp:sphericity} and take $\pi = \mathbf{1}$. Let $O_1, \ldots, O_r$ be the $G(F)$-orbits in $X^+(F)$, all being open and closed. Let $c_i$ be the characteristic function of $O_i \subset X^+(F)$. Take an invariant half-density $|\phi|^{-1/4} |\Omega|^{1/2}$ afforded by Proposition \ref{prop:L-trivializable}. We may assume $|\phi|^{1/4} = |f|^{\lambda_0}$ as in Remark \ref{rem:eta-0}. Then
	\[\begin{tikzcd}[row sep=tiny]
		\CC^k \arrow[r, "\sim"] & \mathcal{N}_{\mathbf{1}}(X^+) \\
		(0, \ldots, \underbracket{\; 1 \;}_{i\text{-th slot}}, \ldots, 0) \arrow[mapsto, r] & \eta_i := c_i |\phi|^{-1/4} |\Omega|^{1/2}.
	\end{tikzcd}\]
	Choose a non-degenerate relative invariant to obtain $X^+ \rightiso \check{X}^+$, so that the $G(F)$-orbits in $X^+(F)$ and $\check{X}^+(F)$ are in bijection and both labeled by $\{1, \ldots, k\}$. Define $\check{\eta}_1, \ldots, \check{\eta}_k$ in this manner.
	
	For all $1 \leq i \leq k$ and $\Re(\lambda) \relgg{X} 0$, we obtain
	\begin{equation*}
		Z_\lambda\left( \eta_i, 1, \xi \right) = \int_{O_i} |f|^{\lambda - \lambda_0} \xi_0 |\Omega|.
	\end{equation*}

	Identifying $\Lambda_{\CC}$ with $\CC^r$ via the basic eigencharacters, $Z_\lambda\left( \eta_i, 1, \xi \right)$ is seen to equal the local zeta integrals $Z_i$ in \cite[\S 1.4, \S 2.2]{Sa89} up to a shift by $\lambda_0$. Moreover, the functional equation in Theorem \ref{prop:LFE} (conditional on Hypothesis \ref{hyp:LFE}) turns out to coincide with that in \cite[p.471, p.477]{Sa89}: our $\gamma$-factor becomes the $\Gamma$-matrices in \textit{loc.\ cit.}
\end{example}

\begin{example}[Godement--Jacquet]
	Let $D$ be a central simple $F$-algebra of dimension $n^2$ and let $G = D^\times \times D^\times$ act on $X := D$ by $x \xrightarrow{(g, h)} h^{-1} x g$. This satisfies Hypothesis \ref{hyp:sphericity} since $X^+ = D^\times$ is a symmetric space. The reduced norm $\mathrm{Nrd}$ is the unique basic relative invariant up to $F^\times$. In parallel, $\mathbf{X}^*_\rho(G)$ is generated by $(g, h) \mapsto \mathrm{Nrd}(h)^{-1} \mathrm{Nrd}(g)$.

	Let $\mathrm{Trd}$ denote the reduced trace and identify $X$ and $\check{X}$ via the perfect pairing $(x, y) \mapsto \mathrm{Trd}(xy)$ on $X \times X$. One can check (see \cite[Lemma 6.4.1]{LiLNM}) that $(G, \check{\rho}, \check{X})$ is isomorphic to $X$ with the action $x \xmapsto{(g, h)} g^{-1} x h$. In fact, $\mathrm{Nrd}$ is non-degenerate, and the induced equivariant isomorphism $X^+ \rightiso \check{X}^+$ is $x \mapsto x^{-1}$ (see \cite[Proposition 6.4.2]{LiLNM}). Note that $\mathrm{Nrd}$ is a basic relative invariant for both $X$ and $\check{X}$, but the eigencharacters are opposite.
	
	The irreducible admissible $\pi$ with $\mathcal{N}_\pi(X^+) \neq \{0\}$ take the form $\sigma \boxtimes \check{\sigma}$. Ditto for $\mathcal{N}_\pi(\check{X}^+)$. The spaces $\mathcal{N}_{\sigma \boxtimes \check{\sigma}}(X^+)$ and $\mathcal{N}_{\sigma \boxtimes \check{\sigma}}(\check{X}^+)$ are spanned respectively by
	\begin{align*}
		\eta_{\Omega}: v \otimes \check{v} & \mapsto \lrangle{\check{v}, \pi(\cdot) v} |\det|^{-n/2} |\Omega|^{1/2}, \\
		\check{\eta}_{\Psi}: v \otimes \check{v} & \mapsto \lrangle{\check{\pi}(\cdot) \check{v}, v} |\det|^{-n/2} |\Psi|^{1/2}.
	\end{align*}

	We write $\xi = \xi_0 |\Omega|^{1/2} \in \Schw(X)$, $\check{\xi} = \check{\xi}_0 |\Psi|^{1/2} \in \Schw(\check{X})$, and let $Z^{\mathrm{GJ}}$ (resp.\ $\gamma^{\mathrm{GJ}}$) be the Godement--Jacquet integrals in \cite[(15.4.3)]{GH11-2} (resp.\ the Godement--Jacquet $\gamma$-factors). For simplicity, assume that $\psi$ is chosen so that the Haar measures $|\Omega|$ and $|\Psi|$ are mutually dual. Our formalism reduces to Godement--Jacquet theory up to a $\frac{1}{2}$-shift, namely:
	\begin{align*}
		Z_\lambda \left( \eta_{\Omega}, v \otimes \check{v}, \xi \right) & = Z^{\mathrm{GJ}}\left( \lambda + \frac{1}{2}, \lrangle{\check{v}, \pi(\cdot) v}, \xi_0 \right), \\
		\check{Z}_\lambda \left( \check{\eta}_{\Psi}, v \otimes \check{v}, \xi \right) & = Z^{\mathrm{GJ}}\left( -\lambda + \frac{1}{2}, \lrangle{\check{\pi}(\cdot) \check{v}, v}, \xi_0 \right), \\
		\gamma(\sigma \boxtimes \check{\sigma}, \lambda, \psi)\left( \check{\eta}_{\Psi} \right) & = \gamma^{\mathrm{GJ}}\left( \lambda + \frac{1}{2}, \sigma, \psi \right) (\eta_\Omega) .
	\end{align*}
	Accordingly, Theorem \ref{prop:LFE} reduces to the usual Godement--Jacquet functional equation. Indeed, the Hypothesis \ref{hyp:LFE} can be verified in this case; see \cite[\S 6.4]{LiLNM} for the case of split $D$.
\end{example}

\section{Beyond spherical spaces}\label{sec:beyond-spherical}

The goal here is to loosen the Hypothesis \ref{hyp:sphericity} at the cost of constraining the admissible representations $\pi$ in question, so that the results from \S\ref{sec:spherical-case} remain valid. The framework will include more cases of arithmetic interest, for example that in \S\ref{sec:SW}.

As in \S\ref{sec:PVS}, we begin with a field $F$ of characteristic zero, a connected reductive $F$-group $G$ together with a reductive prehomogeneous vector space $(G, \rho, X)$ such that $X^+$ is affine. In addition, we also fix a normal connected reductive subgroup $N \lhd G$ and define
\begin{align*}
	\overline{G} & := G/N, \\
	Y^+ & := X^+ /\!/ N \quad \text{(categorical quotient)}.
\end{align*}

\begin{hypothesis}\label{hyp:beyond-sphericity}
	For $(G, \rho, X)$ and $N \lhd G$ as above, we assume that
	\begin{compactitem}
		\item $N$ acts freely on $X^+$;
		\item the quotient map $G(F) \to \overline{G}(F)$ is surjective;
		\item $N$ is semisimple;	% Only used in the Archimedean functional equation.
%		\item $\chi|_N = 1$ for all $\chi \in \mathbf{X}^*_\rho(G)$;
		\item the conditions on $Y^+$ under $\overline{G}$-action as in Hypothesis \ref{hyp:sphericity}, namely: either
		\begin{compactitem}
			\item $F$ is Archimedean and $Y^+$ is a spherical homogeneous $\overline{G}$-space,
			\item $F$ is non-Archimedean, $\overline{G}$ is split and $Y^+$ is a wavefront spherical homogeneous $\overline{G}$-space, or
			\item $F$ is non-Archimedean and $Y^+$ is a symmetric space under $\overline{G}$.
		\end{compactitem}
	\end{compactitem}
\end{hypothesis}

Observe that the case $N = \{1\}$ reduces to Hypothesis \ref{hyp:sphericity}.

\begin{lemma}\label{prop:duality-beyond}
	If $(G, \rho, X)$ satisfies Hypothesis \ref{hyp:beyond-sphericity}, so does $(G, \check{\rho}, \check{X})$.
\end{lemma}
\begin{proof}
	This follows from the fact that $X^+ \simeq \check{X}^+$. %and $\mathbf{X}^*_\rho(G) = \mathbf{X}^*_{\check{\rho}}(G)$.
\end{proof}

\begin{lemma}\label{prop:N-torsor}
	Under Hypothesis \ref{hyp:beyond-sphericity}, the quotient morphism $X^+ \to Y^+$ is an $N$-torsor.
\end{lemma}
\begin{proof}
	Since $\mathrm{char}(F) = 0$, this is a consequence of Luna's slices: see \cite[p.199, Corollary]{MFK94}.
\end{proof}

Hereafter, we assume $F$ is a local field with $\mathrm{char}(F) = 0$. Denote by $|\cdot|$ the normalized absolute value on $F$. The admissible representations $\pi$ of $G(F)$ under consideration are assumed to be trivial on $N(F)$, hence they can also be viewed as admissible representations of $\overline{G}(F)$.

Denote by $p: X^+ \to Y^+$ the quotient morphism. As $p(X^+(F))$ is a finite union of $G(F)$-orbits, it is closed and open in $Y^+(F)$.

\begin{proposition}\label{prop:finiteness-beyond}
	For every admissible representation $\pi$ of $\overline{G}(F)$, the $\CC$-vector space $\mathcal{N}_\pi(X^+)$ is finite-dimensional.
\end{proposition}
\begin{proof}
	Choose a non-vanishing invariant half-density $\alpha$ on $N(F)$. We obtain a linear embedding
	\begin{equation}\label{eqn:finiteness-beyond}\begin{aligned}
		\mathcal{N}_\pi(X^+) & \hookrightarrow \mathcal{N}_\pi(Y^+) \\
		\eta & \mapsto \left[ y \mapsto \begin{cases}
			\eta(x) / \alpha, & \exists x \in X^+(F), \; p(x)=y \\
			0, & \text{otherwise}
		\end{cases}\right].
	\end{aligned}\end{equation}
	Its image is precisely the subspace of intertwining maps $\overline{\eta}: \pi \to C^\infty(p(X^+(F)))$. It remains to apply Theorem \ref{prop:finite-mult} to $Y^+$ and $\overline{G}$.
\end{proof}

Choose basic relative invariants $f_1, \ldots, f_r$ with eigencharacters $\omega_1, \ldots, \omega_r$, and define $|\omega|^\lambda$, $|f|^\lambda$ by \eqref{eqn:omega-lambda}, where $\lambda \in \Lambda_{\CC}$. To all $\pi$ as above, $\eta \in \mathcal{N}_\pi(X^+)$, $v \in V_\pi$ and $\xi \in \Schw(X)$, we define $Z_\lambda(\eta, v, \xi)$ as in Definition \ref{def:zeta}, granting its convergence for $\Re(\lambda) \relgg{X} 0$.

As before, $\mathcal{T} := \left\{ |\omega|^\lambda : \lambda \in \Lambda_{\CC} \right\}$ is isomorphic to $\Lambda_{\CC}$ (resp.\ to $\left(\CC^\times \right)^r$) when $F$ is Archimedean (resp.\ non-Archimedean). The zeta integrals are parametrized by $\mathcal{T}$.

Now we can state the extensions of \S\ref{sec:spherical-case}. The statements are identical to Theorems \ref{prop:convergence-meromorphy} and \ref{prop:LFE}. As before, we omit the issues about continuity, etc.\ for the sake of simplicity; details can be found in \cite{LiLNM, Li19}.

\begin{theorem}\label{prop:convergence-meromorphy-beyond}
	Let $\pi$ be an admissible representation of $\overline{G}(F)$. Under the Hypothesis \ref{hyp:beyond-sphericity}, there exists $\kappa = \kappa(\pi) \in \Lambda_{\R}$, depending solely on $(G, \rho, X)$, $N$ and $\pi$, such that $Z_\lambda(\eta, v, \xi)$ is defined by a convergent integral for all $\eta, v, \xi$ whenever $\Re(\lambda) \relgeq{X} \kappa$, and $Z_\lambda(\eta, v, \xi)$ is holomorphic in $\lambda$ in the interior of that region.
	
	In this case, $\lambda \mapsto Z_\lambda(\eta, v, \xi)$ extends to a meromorphic family indexed by $\mathcal{T}$, for each $(\eta, v, \xi)$. It is a rational family when $F$ is non-Archimedean.
	
	Moreover, in the Archimedean case, the result on convergence holds when $Y^+$ is only \emph{real spherical}, i.e.\ there is an open dense orbit in $Y^+$ under the minimal $\R$-parabolic subgroup of $\overline{G}$.
\end{theorem}

We fix an additive character $\psi$ of $F$ to define $\mathcal{F}_\psi: \Schw(X) \to \Schw(\check{X})$. Denote the zeta integrals for $(G, \check{\rho}, \check{X})$ by $\check{Z}_\lambda$ as before.

\begin{theorem}\label{prop:LFE-beyond}
	Let $\pi$ be an irreducible admissible representation of $\overline{G}(F)$. Under the Hypotheses \ref{hyp:beyond-sphericity} and \ref{hyp:LFE}, there exists a unique meromorphic family of linear maps
	\[ \gamma(\pi, \lambda, \psi): \mathcal{N}_\pi(X^+) \to \mathcal{N}_\pi(X^+), \quad \lambda \in \Lambda_{\CC} \]
	depending on $\psi$, which is indexed by $\mathcal{T}$ and rational in the non-Archimedean case, such that
	\[ \check{Z}_\lambda\left( \check{\eta}, v, \mathcal{F}_\psi \xi \right) = Z_\lambda\left( \gamma(\lambda, \pi, \psi)(\check{\eta}), v, \xi \right) \] 
	for all $\check{\eta} \in \mathcal{N}_\pi(\check{X}^+)$, $v \in V_\pi$, $\xi \in \Schw(X)$, as meromorphic or rational families.
\end{theorem}

The proofs of Theorems \ref{prop:convergence-meromorphy-beyond} and \ref{prop:LFE-beyond} will occupy \S\S\ref{sec:pf-arch}---\ref{sec:pf-nonarch}.

\section{Example: Suzuki--Wakatsuki theory}\label{sec:SW}
The results in this section are proved in \cite{Sa06,SW20}.

To begin with, let $F$ be a field of characteristic zero and consider the data
\begin{compactitem}
	\item $D$: a quaternion $F$-algebra with reduced norm $\Nrd$ and reduced trace $\Trd$;
	\item $G := D^\times \times D^\times \times \GL(2)$, viewed as a connected reductive $F$-group;
	\item $X := D \oplus D$, a $8$-dimensional $F$-vector space;
	\item $\rho: G \to \GL(X)$ is the representation defined by
	\[ (x_1, x_2) \rho(g_1, g_2, g_3) = \left( g_1^{-1} x_1 g_2, g_1^{-1} x_2 g_2 \right) \cdot g_3, \]
	where $g_3 \in \GL(2)$ acts by matrix multiplication.
\end{compactitem}

Note that $\Ker(\rho) = \left\{ (a, b, a b^{-1}) : a, b \in \Gm \right\}$, where we embed $\Gm$ as the centers of $D^\times$ and $\GL(2)$. Hence $\Ker(\rho) \simeq \Gm^2$.

Introduce the variables $v_1, v_2$ and define, for each $x = (x_1, x_2) \in X(F)$, the binary quadratic form
\[ \mathcal{F}_x(v_1, v_2) := \Nrd_{D(v_1, v_2)}\left( x_1 \otimes v_1  + x_2 \otimes v_2 \right) \]
where $\Nrd_{D(v_1, v_2)}$ is the reduced norm for the quaternion $F(v_1, v_2)$-algebra $D \dotimes{F} F(v_1, v_2)$. Then
\begin{equation*}
	\mathcal{F}_{x\rho(g)}(v) = \Nrd(g_1)^{-1} \Nrd(g_2) \cdot \mathcal{F}_x \left(v \cdot {}^t g_3 \right), \quad g = (g_1, g_2, g_3) \in G, \; v = (v_1, v_2).
\end{equation*}

Set
\begin{align*}
	\omega(g) & := \Nrd(g_1)^{-2} \Nrd(g_2)^2 \det(g_3)^2, \quad g=(g_1, g_2, g_3) \in G, \\
	P(x) & := \mathrm{disc}\left( \mathcal{F}_x(v) \right), \quad x \in X,
\end{align*}
so that $\omega \in \mathbf{X}^*(G)$ is trivial on $\Ker(\rho)$, and $P$ is a regular function on $X$ satisfying
\[ P(x\rho(g)) = \omega(g) P(x), \quad (x, g) \in X \times G. \]

\begin{remark}\label{rem:SL2-triple}
	It is sometimes to consider the action of $(SD^\times)^3 \times \Gm$ on $X$, where $SD^\times := \Ker(\Nrd) \lhd D^\times$ and $\Gm$ acts by dilation. The basic eigencharacter then becomes $(g_1, g_2, g_3, t) \mapsto t^2$. See \cite[2.3]{Sa06}.
\end{remark}

\begin{proposition}[see {\cite{Sa06, SW20}}]
	The datum $(G, \rho, X)$ above is a reductive prehomogeneous vector space such that $X^+$ is defined by $P \neq 0$ and $P$ is the unique basic relative invariant (up to $F^\times$). In particular, $X^+$ is affine and $\mathbf{X}^*_\rho(G) = \Z \omega$.
\end{proposition}

Identify $X$ with its dual by the pairing $\lrangle{(x_1, x_2), (x'_1, x'_2)} = \Trd(x_1 x'_1) + \Trd(x_2 x'_2)$. The contragredient representation becomes $\check{\rho}: G \to \GL(X)$ given by
\[ (x_1, x_2) \check{\rho}(g_1, g_2, g_3) = \left( g_2^{-1} x_1 g_1, g_2^{-1} x_2 g_1 \right) \cdot {}^t g_3^{-1}. \]

This setting is used by M.\ Suzuki and S.\ Wakatsuki \cite{SW20} in their study of toric periods on $D^\times$. It also appeared in the works of F.\ Sato \cite{Sa06} for split $D$.

\begin{remark}\label{rem:cube}
	The prehomogeneous vector space $(G, \rho, X)$ arises in Bhargava's setting \cite{Bha04} in the following way. Take the split quaternion $F$-algebra $D = \mathrm{M}_2$, so $D^\times = \GL(2)$. Recall that $\Nrd = \det$ for split $D$. Regard $F^2$ as the space of row vectors, and let $\GL(2)^3$ act on the $F$-vector space $W := F^2 \otimes F^2 \otimes F^2$ in the natural way. Expand the last tensor slot into $F \oplus F$ to obtain
	\[ W = (\underbracket{F^2 \otimes F^2}_{\text{slots 1,2}} ) \oplus (\underbracket{F^2 \otimes F^2}_{\text{slots 1,2}}). \]
	Then $\{1\} \times \{1\} \times \GL(2)$ acts by matrix multiplication $(x_1, x_2) \mapsto (x_1, x_2) \cdot g_3$, for all $x_1, x_2 \in F^2 \otimes F^2$ and $g_3 \in \GL(2)$.
	
	Now fix a non-degenerate pairing $\lrangle{\cdot, \cdot}$ on $F^2$ to get
	\begin{align*}
	F^2 \otimes F^2 & \rightiso \End_F(F^2) = \mathrm{M}_2 \\
	u \otimes v & \mapsto \lrangle{u, \cdot} v.
	\end{align*}
	Then $\GL(2) \times \GL(2) \times \{1\}$ acts on each $\mathrm{M}_2$ factor via
	\[ x = \lrangle{u, \cdot} v \xmapsto{(g_1, g_2)} \lrangle{u g_1, \cdot} (v g_2) = \lrangle{u, (\cdot) {}^\dagger g_1} (v g_2), \]
	and the last term is ${}^\dagger g_1 x g_2$ by our conventions; here $g_1 \mapsto {}^\dagger g_1$ is the transpose with respect to $\lrangle{\cdot, \cdot}$. All in all, $\GL(2)^3$ acts on $W \simeq \mathrm{M}_2 \oplus \mathrm{M}_2$ via
	\[ (x_1, x_2) \xmapsto{(g_1, g_2, g_3)} \left( {}^\dagger g_1 x_1 g_2, {}^\dagger g_1 x_2 g_2 \right) \cdot g_3. \]
	This coincides with our original setting, up to a twist $g_1 \mapsto {}^\dagger g_1^{-1}$.
	
	We may also consider the $\SL(2)^3 \times \Gm$-action on $W$, where $\Gm$ acts by dilation (see Remark \ref{rem:SL2-triple}). If $\lrangle{\cdot, \cdot}$ is the pairing corresponding to $\bigl( \begin{smallmatrix} & 1 \\ -1 & \end{smallmatrix} \bigr)$ such that the symplectic group equals $\SL(2)$, then ${}^\dagger g_1^{-1} = g_1$ for all $g_1 \in \SL(2)$. The resulting prehomogeneous vector space is therefore that of Remark \ref{rem:SL2-triple}. This prehomogeneous vector space appeared in Bhargava's work \cite{Bha04}, where its integral structure is studied in depth.
\end{remark}

Let us show that the data above fit into the framework in \S\ref{sec:beyond-spherical}.

\begin{proposition}
	The data $(G, \rho, X)$ and $N := \SL(2) \lhd G$ (embedded in the third factor) of Suzuki--Wakatsuki satisfy the requirements in Hypothesis \ref{hyp:beyond-sphericity}. Specifically, $Y^+ := X^+ /\!/ N$ is a symmetric space under $\overline{G} = D^\times \times D^\times \times \Gm$.
\end{proposition}
\begin{proof}
	The prehomogeneity has just been stated. The map $G(F) \to \overline{G}(F)$ is surjective since $\det: \SL(2, F) \to F^\times$ is. To show that $Y^+$ is a symmetric space, one can pass to an extension of $F$ to ensure that $D$ is split, and then apply the explicit descriptions in \cite[p.81]{Sa06} or \cite{SW20}.
	
	To show that $N$ acts freely on $X^+$, consider $(x_1, x_2) \in X^+(F)$. Hence $x_1$, $x_2$ are $F$-linearly independent in $D$, otherwise $\mathcal{F}_x(v_1, v_2)$ would be degenerate. Therefore, for all $g_3 \in \GL(2, F)$,
	\[ (x_1, x_2) \cdot g_3 = (x_1, x_2) \iff g_3 = 1. \]
	This implies the freeness of $N$-action since $F$ can be replaced by any field extension.
\end{proof}

In order to apply the Theorem \ref{prop:LFE-beyond} when $F$ is a non-Archimedean local field, we show the following
\begin{proposition}
	Our triplet $(G, \rho, X)$ satisfies Hypothesis \ref{hyp:LFE}.
\end{proposition}
\begin{proof}
	We use the description in \cite{SW20} of stabilizers of points of $(\partial X)(F)$. Let us begin with the case of split $D$. For $1 \leq i, j \leq 2$, write
	\[ E_{ij} = \begin{tikzpicture}[baseline]
		\matrix (M) [matrix of math nodes, left delimiter=(, right delimiter=)] {
			& \vdots & \\
			\cdots & 1 & \cdots \\
			& \vdots & \\
		};
		\node[right=2.5em] at (M-2-3) {row $i$};
		\node[below=1em] at (M-3-2) {column $j$ };
	\end{tikzpicture} \]
	and define the following parabolic subgroups of $G$
	\begin{align*}
		B & := \twomatrix{*}{*}{}{*} \times \twomatrix{*}{*}{}{*} \times \twomatrix{*}{*}{}{*} , \\
		Q & := D^\times \times D^\times \times \twomatrix{*}{*}{}{*}.
	\end{align*}
	
	An explicit set of representatives $y$ for the $G(F)$-orbits in $(\partial X)(F)$, their stabilizers $H := \Stab_G(y)$ and the parabolic subgroup $P$ required in Hypothesis \ref{hyp:LFE} are tabulated in Table \ref{tab:SW}.
	\begin{table}[h]\centering
		\[\def\arraystretch{1.2}\begin{array}{c|c|c}
			y & H & P \\[2pt] \hline & & \\
			(0, E_{12}) & \left\{ \left( \twomatrix{a}{*}{}{*}, \twomatrix{*}{*}{}{b}, \twomatrix{*}{*}{}{c} \right)  : a = bc \right\} & B \\
			(0, 1) & \left\{ \left( g, h, \twomatrix{*}{*}{}{c} \right)  : g = c \cdot h \right\} & Q \\
			(E_{11}, E_{12}) & \simeq \text{to above} & \simeq Q \\
			(E_{12}, E_{22}) & \simeq \text{to above} & \simeq Q \\
			(E_{12}, 1) & \left\{ \left( \twomatrix{a_1}{a_2}{}{a_3}, \twomatrix{a_1}{a_2}{}{a_3}  \twomatrix{\frac{1}{c_3}}{\frac{-c_2}{c_1 c_3}}{}{\frac{1}{c_3}}, \twomatrix{c_1}{c_2}{}{c_3} \right)  : a_1 c_3 = a_3 c_1 \right\} & B \\
			(0, 0) & G & G
		\end{array}\]
		\caption{Representatives of $G(F)$-orbits in $(\partial X)(F)$}
		\label{tab:SW}
	\end{table}
	
	Specifically, the isomorphisms of $H$ and $P$ among the rows with $y = (0, 1)$, $(E_{11}, E_{12})$ or $(E_{12}, E_{22})$ are realized as follows. Identify $X$ with $(F^2)^{\otimes 3}$ as in Remark \ref{rem:cube}. The three $G(F)$-orbits are then related by permuting the tensor slots of $(F^2)^{\otimes 3}$ and the factors of $G = \GL(2)^3$ accordingly. The author is indebted to S.\ Wakatsuki for this observation.

	In each case we pose $M := P/U_P$, $H_M := H/U_P$, and let $\omega^\flat$ be the character on $Z_M \cap H_M$ induced by $\omega$. Then
	\[\def\arraystretch{1.2}\begin{array}{c|c|c}
		y & Z_M \cap H_M & \omega^\flat \\[2pt] \hline & & \\
		(0, E_{12}) & \left\{ (a,a',b',b,c',c) : a = bc \right\} & \left( \frac{b'c'}{a'} \right)^2 \\
		(0, 1) & \left\{ (u,v,c',c) : u=cv \right\} & \left( \frac{c'}{c} \right)^2 \\
		(E_{12}, 1) & \left\{ (a_1, a_3, c_1, c_3) : a_1 c_3 = a_3 c_1 \right\} & \left(\frac{c_1}{c_3}\right)^2 \\
		(0, 0) & \left\{ (u,v,w) \in \Gm^3 \right\} & \left(\frac{vw}{u}\right)^4
	\end{array}\]
	where $a, a'$, $u, v$, etc.\ denote elements of $\Gm$. Note that we omitted the cases of $(E_{11}, E_{12})$ and $(E_{12}, 1)$ because they can be obtained from the case of $y = (0, 1)$, by applying some automorphisms of $G$.

	Thus in each case, $Z_M \cap H_M$ is a split $F$-torus, and the characters $|\omega^\flat|^\lambda$ (for $\lambda \in \CC$) form a complex torus of positive dimension. This verifies Hypothesis \ref{hyp:LFE}.

	When $D$ is a quaternion division $F$-algebra, there are only two $G(F)$-orbits in $(\partial X)(F)$, represented by $y = (0, 1)$ and $y = (0, 0)$. The calculations above still validate the Hypothesis \ref{hyp:LFE}.
\end{proof}

\begin{remark}\label{rem:Y-quotient}
	In the case of split $D$, the categorical quotient $Y := X /\!/ N$ is described explicitly in \cite[pp.80--81]{Sa06}: it is isomorphic to the space of pairs $(y_1, y_2)$ of symmetric bilinear forms on $F^2$ such that $\mathrm{disc}(y_1) = \mathrm{disc}(y_2)$. The categorical quotient $Y^+$ is isomorphic to its open subset defined by $\mathrm{disc}(y_1) \neq 0$. Observe that $Y$ is conic, thus non-smooth.
\end{remark}

Finally, we sketch the global zeta integral which is the main concern of \cite{SW20}. Let $F$ be a number field and set $H := G/\Ker(\rho)$. Let $\phi_i$ be an $L^2$-automorphic form on $D^\times(\A_F)$ with trivial central character ($i=1,2$). Let $\Schw(X(\A_F))$ be the adélic Schwartz--Bruhat space on $X$; half-densities can be avoided in the global case by choosing any $F$-rational algebraic volume form on $X$. The zeta integrals in question is
\[ Z_\lambda(\phi_1, \phi_2, \xi) := \int_{H(F) \backslash H(\A_F)} \phi_1(g_1) \phi_2(g_2) \theta_\xi(h) |\omega|^\lambda \dd h \]
where
\begin{compactitem}
	\item $\xi \in \Schw(X(\A_F))$ and $\lambda \in \CC$ satisfies $\Re(\lambda) \gg 0$;
	\item we take $(g_1, g_2, g_3)$ to be any representative of $h$ in $G(\A_F)$;
	\item $\dd h$ is a Haar measure on $H(\A_F)$ (eg.\ the Tamagawa measure);
	\item $\theta_\xi(h) := \sum_{x \in X^+(F)} \xi(x \rho(h))$ is the ``$\theta$-function'' attached to $\xi$.
\end{compactitem}
For the dual version $\check{Z}_\lambda(\phi_1, \phi_2, \xi)$ associated with $\check{\rho}$, one replaces $|\omega|^\lambda$ by $|\omega|^{-\lambda}$ and $\theta_\xi$ by $\check{\theta}_\xi: h \mapsto \sum_{x \in X^+(F)} \xi(x \check{\rho}(h))$.

We refer to \cite{SW20} for further properties of $Z_\lambda$ and $\check{Z}_\lambda$, such as the convergence when $\Re(\lambda) \gg 0$.

Also note that $\phi_1 \otimes \phi_2 \otimes |\omega|^\lambda$ factors into an automorphic form on $\overline{G}(\A_F)$. By assuming that the form on $\overline{G}(\A_F)$ is cuspidal, all these constructions are compatible with the framework set up in \cite[Chapter 8]{LiLNM}, provided that we go beyond the spherical setting by allowing the situation of Hypothesis \ref{hyp:beyond-sphericity}.

\section{Some invariant theory}\label{sec:invariant-theory}
This section serves as a preparation for the proof of Theorem \ref{prop:LFE-beyond} in the Archimedean case. The following arguments are adapted from \cite{Kn98}. Throughout this section, we consider a triplet $(G, \rho, X)$ and $N \lhd G$ subject to Hypothesis \ref{hyp:beyond-sphericity}, with $F = \CC$. Define $Y := X /\!/ N$, so that $Y^+$ is the open $\overline{G}$-orbit in $Y$.

Fix a Borel pair $B \supset T$ for $G$ whose image in $\overline{G}$ is also a Borel pair $\overline{B} \supset \overline{T}$. Denote by $\mathbf{X}^*(\overline{T})^+ \subset \mathbf{X}^*(\overline{T})$ the subset dominant weights with respect to $\overline{B}$. In what follows, the highest/lowest weight vectors are taken relative to $\overline{B}$ or $B$.

Denote the $G$-action on $\CC[X]$ as $f \mapsto {}^g f$. Define the $\overline{G}$-module
\[ \mathcal{P} := \CC[X]^N = \CC[Y]. \]
Since $Y^+$ is spherical, $\mathcal{P}$ is multiplicity-free. Specifically, there is a submonoid $\Lambda(Y)^+ \subset \mathbf{X}^*(\overline{T})^+$ such that
\begin{equation}\label{eqn:PD-decomp}
	\mathcal{P} = \bigoplus_{\lambda \in \Lambda(Y)^+} \mathcal{P}_\lambda \quad \text{as $\overline{G}$-modules},
\end{equation}
where $\mathcal{P}_\lambda$ is the simple $\overline{G}$-module with lowest weight $-\lambda$.

\begin{proposition}[see {\cite[3.2]{Kn98}}]\label{prop:Lambda-Y}
	There exist linearly independents weights $\lambda_1, \ldots, \lambda_s \in \mathbf{X}^*(\overline{T})$ such that $\Lambda(Y)^+ = \sum_{i=1}^s \Z_{\geq 0} \lambda_i$.
	
	For each $i$, let $f_i \in \mathcal{P}_{\lambda_i}$ be a highest weight vector. Then every highest weight vector in $\mathcal{P}$ is proportional to $f_1^{a_1} \cdots f_s^{a_s}$ for some $(a_1, \ldots, a_s) \in \Z_{\geq 0}^s$.
\end{proposition}
\begin{proof}
	For every $\lambda \in \Lambda(Y)^+$, choose a highest weight vector $f_\lambda \in \mathcal{P}_\lambda$, which is unique up to $\CC^\times$. Let $\lambda_1, \lambda_2, \ldots \in \Lambda$ be such that $f_{\lambda_i}$ is an irreducible polynomial for all $i$. Given $\lambda \in \Lambda$, factorize $f_\lambda$ into irreducibles in $\CC[X]$ as $f_1 \cdots f_d$. For every $b \in B$, we obtain ${}^b f_\lambda = {}^b f_1 \cdots {}^b f_d$. Since ${}^b f_\lambda \in \CC^\times f_\lambda$, the $B$-action permutes $f_1, \ldots, f_d$ up to $\CC^\times$. By connectedness, we conclude that every $f_i$ is a relative $B$-invariant. Similarly, they are relative $N$-invariants, thus $N$-invariant. Hence $f_1, \ldots, f_d$ are highest weight vectors in $\mathcal{P}$.
	
	It follows that $\lambda_1, \lambda_2, \ldots$ span $\Lambda(Y)^+$. Moreover, distinct monomials in $f_{\lambda_1}, f_{\lambda_2}, \ldots$ must have different weights, otherwise we will get multiplicities in $\mathcal{P}$, in view of the unique factorization in $\CC[X]$. Hence $\lambda_1, \lambda_2, \ldots$ are $\Z$-linearly independent in $\mathbf{X}^*(\overline{T})$, thus finite.
\end{proof}

In view of the linear reductivity of $N$, the contragredient $\overline{G}$-module of $\mathcal{P}$ is
\[ \mathcal{D} := \CC[\check{X}]^N = \bigoplus_{\lambda \in \Lambda(Y)^+} \mathcal{D}_\lambda, \]
where $\mathcal{D}_\lambda$ is the simple $\overline{G}$-module of highest weight $\lambda$. Hence
\[ \left( \mathcal{P} \otimes \mathcal{D} \right)^{\overline{G}} = \bigoplus_{\lambda \in \Lambda(Y)^+} \left( \mathcal{P}_\lambda \otimes \mathcal{D}_\lambda \right)^{\overline{G}}. \]

Let $\tilde{B} \subset G$ be the preimage of $\overline{B} \subset \overline{G}$, so that $\tilde{B} \supset B$. Define
\[ \mathfrak{a}^* := (\bigoplus_{i=1}^s \Z \lambda_i) \otimes \CC = \bigoplus_{i=1}^s \CC \lambda_i. \]

Let $X_0 \subset X$ be the preimage of the open $\overline{B}$-orbit $Y_0$ in $Y$. Every highest weight vector $f \in \mathcal{P}$ is invertible on $X_0$. The morphism
\begin{equation*}
	\phi_f = f^{-1} \dd f : X_0 \to \check{X}
\end{equation*}
is $\tilde{B}$-equivariant and depends only on the weight of $f$. More generally, for $\chi = \sum_{i=1}^s a_i \lambda_i \in \mathfrak{a}^*$, we define the $\tilde{B}$-equivariant morphism
\[ \phi_\chi := \sum_{i=1}^s a_i \phi_{f_i}. \]

\begin{lemma}\label{prop:regularity}
	For $\lambda \in \Lambda(Y)^+$ in general position, any highest weight vector $f \in \mathcal{P}_\lambda$ is a non-degenerate relative invariant for the prehomogeneous vector space $(\tilde{B}, \rho|_{\tilde{B}}, X)$. Consequently, $\phi_f(X_0) \subset \check{X}^+$.
\end{lemma}
\begin{proof}
	We may assume $f = \prod_{i=1}^s f_i^{a_s}$ where $f_1, \ldots$ are as in Proposition \ref{prop:Lambda-Y}. It is known that $f$ is non-degenerate if and only if $\phi_f: X^+ \to \check{X}$ is dominant, if and only if the differential of
	\[ \phi_f = \sum_{i=1}^s a_i \phi_{f_i} \]
	at some (equivalently, any) $x \in X_0$ is an isomorphism. Given $x$, this is an open, algebraic condition in $(a_1, \ldots, a_s) \in \CC^s$.  Since $(G, \rho, X)$ is regular, there exists some $(a_1, \ldots, a_s)$ such that $f$ is non-degenerate, so the condition holds for $(a_1, \ldots, a_s)$ in general position.
	
	If $f$ is non-degenerate, then $\phi_f(X_0)$ is contained in the open $\tilde{B}$-orbit in $\check{X}$, which is in turn inside $\check{X}^+$.
\end{proof}

Let $\Delta: G \to G \times G$ be the diagonal embedding. Unless otherwise specified, subgroups of $G$ will act on $X \times \check{X}$ through $\Delta$. We also identify $X \times \check{X}$ with the cotangent bundle $T^* X$.

\begin{lemma}[cf.\ {\cite[4.1]{Kn98}}]\label{prop:moment}
	For every $x \in X_0$, define $\mathfrak{a}^*(x) := \left\{ (x, \phi_\chi(x)) : \chi \in \mathfrak{a}^* \right\} \subset X \times \check{X}$. Then
	\begin{enumerate}[(i)]
		\item $\mathfrak{a}^*(xb) = \mathfrak{a}^*(x) b$ for all $b \in \tilde{B}$;
		\item $\chi \mapsto (x, \phi_\chi(x))$ induces an isomorphism $\mathfrak{a}^* \rightiso \mathfrak{a}^*(x)$;
		\item the constructible subset $\mathfrak{a}^*(x) \cdot (\Delta(G) \cdot (N \times N))$ is dense in $X \times \check{X}$.
	\end{enumerate}
\end{lemma}
\begin{proof}
	In contrast with the elementary proof given in \textit{loc.\ cit.} for $N = \{1\}$, we will deduce the required properties from the general theory in \cite{Kn94a}.
	
	The assertion (i) follows directly from the $\tilde{B}$-equivariance of $\phi_\chi$.
	
	Consider (ii). Firstly, $\Lambda(Y)^+$ generates the lattice of $\overline{B}$-weights in $\CC(Y^+)$ (see for example \cite[Proposition 5.14]{Ti11}); we can actually extract a $\Z$-basis as in Proposition \ref{prop:Lambda-Y}.

	Note that $\phi_\chi$ gives a section of $T^* Y_0$ for every $\chi \in \mathfrak{a}^*$. Since $Y$ is spherical, by \cite[p.314]{Kn94a} there is a nonempty open subset $Y'_0 \subset Y_0$ over which $\phi_{f_1}, \ldots, \phi_{f_s}$ are linearly independent; by the $\tilde{B}$-equivariance of $\phi_{f_i}$, we may take $Y'_0 = Y_0$. This proves (ii).

	Next, consider
	\[ \psi^*: Y_0 \times \mathfrak{a}^* \to T^* Y_0 \subset T^* Y^+, \quad (y, \chi) \mapsto (y, \phi_\chi(y)). \]
	Let $C := \psi^*(Y_0 \times \mathfrak{a}^*)$. Since $Y^+$ is affine, \cite[3.1 Lemma + 3.2 Theorem]{Kn94a} says that the constructible subset $C \cdot G$ is dense in $T^* Y^+$.
	
	Recall that $X^+ \to Y^+$ is an $N$-torsor. In particular $X^+ \dtimes{Y^+} T^* Y^+ \to T^* X^+$ is a sub-bundle. In order to obtain (iii), it suffices to show that the $\{1\} \times N$ action on $X \times \check{X}$ ``fills up the gap'' between $X^+ \dtimes{Y^+} T^*_y Y^+$ and $T^* X^+$.
	
	Indeed, for $x \in X_0$ and $\chi \in \bigoplus_{i=1}^s \Z\lambda_i$ in general position, $\phi_\chi(x) \in \check{X}^+$ by Lemma \ref{prop:regularity}. Thus the same holds for $(x, \chi) \in X_0 \times \mathfrak{a}^*$ in general position, and it remains to recall that $\check{X}^+$ is an $N$-torsor by Lemma \ref{prop:duality-beyond} and \ref{prop:N-torsor}.
\end{proof}

\begin{corollary}\label{prop:c-bar}
	For any $x \in X_0$, restriction to $\mathfrak{a}^*(x)$ defines a $\CC$-linear map $\overline{c}: (\mathcal{P} \otimes \mathcal{D})^{\overline{G}} \to \CC[\mathfrak{a}^*]$, given by $h \mapsto h|_{\mathfrak{a}^*(x)}$ composed with the $\mathfrak{a}^* \rightiso \mathfrak{a}^*(x)$ in Lemma \ref{prop:moment} (ii). It does not depend on $x$.
\end{corollary}
\begin{proof}
	The independence of $x$ follows from Lemma \ref{prop:moment} (i). Since $h$ is $\Delta(G) \cdot (N \times N)$-invariant (as an element of $\mathcal{P} \otimes \mathcal{D}$), the injectivity follows from Lemma \ref{prop:moment} (iii).
\end{proof}

\section{Invariant differential operators}\label{sec:D}
We conserve the notations from \S\ref{sec:invariant-theory}. The following arguments are still based on \cite{Kn98}.

Unless otherwise specified, differential operators will always mean algebraic differential operators. Our differential operators on $Y$ (possibly singular) are certain linear operators on $\CC[Y] = \mathcal{P}$, defined in the Grothendieck style as in \cite[1.1.6 (i)]{BB93}. These operators form a $\CC$-algebra $\mathcal{D}(Y)$ with $\overline{G}$-action, and restriction gives an equivariant injection $\mathcal{D}(Y) \to \mathcal{D}(Y^+)$; its image equals $\left\{ D \in \mathcal{D}(Y^+) : D\mathcal{P} \subset \mathcal{P} \right\}$.

\begin{lemma}
	We have $\mathcal{D}(Y)^{\overline{G}} \rightiso \mathcal{D}(Y^+)^{\overline{G}}$.
\end{lemma}
\begin{proof}
	Since $\CC[Y^+]$ is multiplicity-free as a $\overline{G}$-module, and $\mathcal{D}(Y^+)^{\overline{G}}$ preserves its simple summands, it follows that $\mathcal{D}(Y^+)^{\overline{G}}$ preserves the $\overline{G}$-submodule $\mathcal{P}$.
\end{proof}

\begin{definition}
	Let $\lambda \in \Lambda(Y)^+$. By multiplicity-one property of $\mathcal{P}$, every $D \in \mathcal{D}(Y)^{\overline{G}}$ acts on $\mathcal{P}_\lambda$ by a scalar, denoted as $c_D(\lambda)$.
\end{definition}

This gives a homomorphism of $\CC$-algebras
\[ \mathcal{D}(Y)^{\overline{G}} \to \left\{ \text{functions}\; \Lambda(Y)^+ \to \CC \right\} \]
mapping $D$ to $[\lambda \mapsto c_D(\lambda)]$. It is injective; in particular $\mathcal{D}(Y)^{\overline{G}}$ is commutative.

\begin{lemma}
	The map above factors as $c: \mathcal{D}(Y)^{\overline{G}} \to \CC[\mathfrak{a}^*]$.
\end{lemma}
\begin{proof}
	Same as the first part of \cite[4.4 Corollary]{Kn98}, which is based on
	\[ c_D(\sum_{i=1}^s a_i \lambda_i) = \frac{D(\prod_{i=1}^s f_i^{a_i})}{\prod_{i=1}^s f_i^{a_i}} \]
	and the general result \cite[4.3 Lemma]{Kn98}.
\end{proof}

These discussions are related to \S\ref{sec:invariant-theory} in the following way. There is an evident $G$-equivariant linear map $\CC[X] \otimes \CC[\check{X}] \to \mathcal{D}(X)$: the elements of $\CC[\check{X}]$ act as differential operators with constant coefficients, whilst those of $\CC[X]$ act as differential operators of order zero.

\begin{lemma}\label{prop:diff-op}
	The elements of $\mathcal{P} \otimes \mathcal{D}$ induce elements of $\mathcal{D}(Y)$ through their action on $\mathcal{P}$. It induces a linear map $\left( \mathcal{P} \otimes \mathcal{D} \right)^{\overline{G}} \to \mathcal{D}(Y)^{\overline{G}}$.
\end{lemma}
\begin{proof}
	These elements give rise to $\CC$-linear endomorphisms of $\CC[X]$ preserving $\mathcal{P}$. According to the Grothendieck-style definition of differential operators, namely by taking commutators with the endomorphisms from $\mathcal{P} \subset \CC[X]$, we see that they restrict to differential operators on $Y$. These manipulations are compatible with $\overline{G}$-actions.
\end{proof}

In contrast with \cite{Kn98}, we say nothing about the injectivity or surjectivity of $(\mathcal{P} \otimes \mathcal{D})^{\overline{G}} \to \mathcal{D}(Y)^{\overline{G}}$. We are only interested in its image.

For a differential operator $D$ on $X$, denote by $\sigma_D \in \CC[X \times \check{X}]$ its principal symbol. Note that when $D \in (\mathcal{P} \otimes \mathcal{D})^{\overline{G}}$, so does $\sigma_D$.

\begin{corollary}[cf.\ the second part of {\cite[4.4 Corollary]{Kn98}}]\label{prop:top-term}
	For every $D \in (\mathcal{P} \otimes \mathcal{D})^{\overline{G}}$, the top homogeneous component of $c_D$ equals $\overline{c}_{\sigma_D}$ (see Corollary \ref{prop:c-bar}).
\end{corollary}
\begin{proof}
	As in the cited proof, this is based on \cite[4.3 Lemma (b)]{Kn98}, applied to $X$. The key ingredient is to show that
	\[ \sigma_D ( \underbracket{a_1 \phi_{f_1}(x) + \cdots + a_s \phi_{f_s}(x)}_{\in \mathfrak{a}^*(x)} \neq 0 \]
	for $x \in X_0$ and $\sum_{i=1}^s a_i \lambda_i \in \Lambda(Y)^+$ when both are in general position. The non-vanishing follows readily from the fact that $\sigma_D \in (\mathcal{P} \otimes \mathcal{D})^{\overline{G}}$ and Lemma \ref{prop:moment}.
\end{proof}

Let $\overline{\rho}$ be the half-sum of positive roots for $\overline{T} \subset \overline{B} \subset \overline{G}$, and let $W$ be the Weyl group for $(\overline{G}, \overline{T})$. View $\overline{\rho}$ as an element of $\overline{\mathfrak{t}}^*$. Observe that $\mathfrak{a}^* \hookrightarrow \mathfrak{t}^*$ canonically, hence $\mathfrak{a}^* + \overline{\rho}$ is an affine subspace of $\mathfrak{t}^*$.

Define the injective homomorphism
\[\begin{tikzcd}[row sep=tiny]
	\mathrm{HC}: \mathcal{D}(Y)^{\overline{G}} \arrow[hookrightarrow, r] & \CC{[\mathfrak{a}^* + \overline{\rho}]} \\
	D \arrow[mapsto, r] & {\left[ \chi \mapsto c_D(\chi - \overline{\rho})\right] }
\end{tikzcd}\]

We record the following instance of Knop's Harish-Chandra isomorphism for $Y^+$. Let $\mathcal{Z}(\overline{\mathfrak{g}})$ denote the center of $U(\overline{\mathfrak{g}})$.

\begin{theorem}[F.\ Knop {\cite[4.8 Theorem]{Kn98}}; see also {\cite{Kn94b}}]
	There exists a subgroup $W_Y$ of $W$ which acts as a reflection group on $\mathfrak{a}^*$, stabilizes $\mathfrak{a}^* + \overline{\rho}$ and makes the following diagram commutative
	\[\begin{tikzcd}[column sep=small]
		\mathcal{Z}(\overline{\mathfrak{g}}) \arrow[rr, "\sim"'] \arrow[d] & & {\CC[\overline{\mathfrak{t}}^*]^W} \arrow[d] \\
		\mathcal{D}(Y^+)^{\overline{G}} \arrow[equal, r] & \mathcal{D}(Y)^{\overline{G}} \arrow[r, "\mathrm{HC}", "\sim"'] & {\CC[\mathfrak{a}^* + \overline{\rho}]^{W_Y} }
	\end{tikzcd}\]
	where the top row is the usual Harish-Chandra isomorphism, the leftmost arrow comes from $\overline{G}$-action, and the rightmost one is restriction.
\end{theorem}
\begin{proof}
	The arguments in \cite{Kn98} carry over verbatim. The only input to check is that $\mathcal{D}(Y^+)^{\overline{G}}$ is a polynomial algebra. Since $Y^+$ is spherical, the property comes from \cite{Kn94b}; in fact, we should get the same Harish-Chandra isomorphism as in \cite{Kn94b}.
\end{proof}

Note that every $\lambda \in \mathbf{X}^*_\rho(G) \otimes \CC \subset \mathfrak{a}^*$ is $W_Y$-invariant, thus the translation by $\lambda$ makes sense on $(\mathfrak{a}^* + \overline{\rho}) /\!/ W_Y$.

\begin{corollary}[Cf.\ {\cite[Proposition 5.7]{Li19}}]\label{prop:HC-twist}
	Let $D \in \mathcal{D}(Y)^{\overline{G}}$. Let $f \in \CC(X)$ be a relative invariant with eigencharacter $\lambda \in \mathbf{X}^*_\rho(G)$. For all $s \in \Z$, the differential operator $D_{f, s} := f^{-s} \circ D \circ f^s$ belongs to $\mathcal{D}(Y^+)^{\overline{G}} = \mathcal{D}(Y)^{\overline{G}}$, and
	\[ \mathrm{HC}(D_{f, s})(x) = \mathrm{HC}(D)(x - s\lambda) \]
	for all $x \in (\mathfrak{a}^* + \overline{\rho}) /\!/ W_Y$.
\end{corollary}
\begin{proof}
	Same as in \textit{loc.\ cit.}, namely by inspecting the effect of $D_{f, s}$ on $\mathcal{P}_\mu$, for $\mu \in \Lambda(Y)^+$ that are sufficiently positive.
\end{proof}

\section{Proofs in the Archimedean case}\label{sec:pf-arch}
Let $(G, \rho, X)$ and $N \lhd G$ be as in Hypothesis \ref{hyp:beyond-sphericity}. We are going to establish the theorems in \S\ref{sec:beyond-spherical} for Archimedean $F$. By restriction of scalars as in \cite[\S 3.4]{Li19}, we reduce to the case $F = \R$.

\begin{lemma}\label{prop:max-cpt}
	There exist maximal compact subgroups $K \subset G(\R)$ and $\overline{K} \subset \overline{G}(\R)$, such that the preimage $L$ of $\overline{K}$ in $G(\R)$ satisfies $L^\circ \subset K^\circ \cdot N(\R)^\circ$. Here we write $L^\circ$ for the identity connected component of $L$, for any real Lie group $L$.
\end{lemma}
\begin{proof}
	This reduces to an assertion about reductive Lie algebras over $\R$. Thus we may assume $G = \overline{G} \times N$, and take the Cartan involutions of $G$, $\overline{G}$ and $N$ in a compatible way.
\end{proof}

\begin{proof}[Proof of Theorem \ref{prop:convergence-meromorphy-beyond}]
	Fix $\Omega \in \topwedge \check{X} \smallsetminus \{0\}$ and $\phi$ as in Proposition \ref{prop:L-trivializable}. Write every $\eta \in \mathcal{N}_\pi(X^+)$ as $\eta = \eta_0 |\phi|^{-1/4} |\Omega|$ as in Remark \ref{rem:eta-0}.
	
	Let $\pi$ be an SAF representation of $G(\R)$ that is trivial on $N(\R)$, and fix $\eta \in \mathcal{N}_\pi(X^+)$. By Remark \ref{rem:eta-0},
	\[ Z_\lambda(\eta, v, \xi) = \int_{X^+(\R)} \eta_0(v) |f|^{\lambda - \lambda_0} \xi_0 |\Omega|, \quad \xi_0 \in \Schw_0(X). \]

	As in \cite[\S 4.1]{Li19}, the proof of convergence for $\Re(\lambda) \relgg{X} 0$ reduces to the following fact: there exist a continuous semi-norm $q: V_\pi \to \R_{\geq 0}$ and a Nash function $p: X^+(\R) \to \R_{\geq 0}$ such that
	\[ |\eta_0(v)(x)| \leq q(v) p(x) \]
	for all $v \in V_\pi$ and $x \in X^+(\R)$. We refer to \cite[\S 3]{AG08} for the basic properties of Nash manifolds and Nash functions on them.
	
	Since
	\begin{compactenum}[(i)]
		\item $\pi$ factors through $\overline{G}(\R)$,
		\item $\eta_0$ factors through $\overline{\eta}_0 \in \Hom_{\overline{G}(\R)}(\pi, C^\infty(Y^+(\R); \CC))$ (see the proof of Proposition \ref{prop:finiteness-beyond}), and that
		\item every Nash function on $Y^+(\R)$ pulls back to a Nash function on $X^+(\R)$,
	\end{compactenum}
	it suffices to establish the same property for $\overline{\eta}_0$ and $Y^+(\R)$. Since $Y^+$ is spherical, the required estimates are all contained in \cite[Proposition 4.1]{Li19}. To show that the region of convergence is uniform in $\eta$, one applies Proposition \ref{prop:finiteness-beyond}. Note that the ingredients in \textit{loc.\ cit.} only require $Y^+$ to be real spherical.

	For the meromorphic continuation of zeta integrals, we follow \cite[\S 4.2]{Li19}. Choose maximal compact subgroups $K \subset G(\R)$ and $\overline{K} \subset \overline{G}(\R)$ as in Lemma \ref{prop:max-cpt}. To achieve the meromorphic continuation, it suffices to show that $\eta_0(v)$ generates a holonomic $\mathscr{D}_{X^+_{\CC}}$-module on $X^+_{\CC}$, for every $K$-finite $v \in V_\pi$. Here $\mathscr{D}$ denotes the sheaf of algebras of algebraic differential operators.

	The choice of $K$ and $\overline{K}$ implies that $v$ is $\overline{K}$-finite. Hence \cite[Proposition 4.2]{Li19} asserts that $\mathscr{D}_{Y^+_{\CC}} \cdot \overline{\eta}_0(v)$ is holonomic on $Y^+_{\CC}$. Since $\eta_0(v)$ is the pull-back of $\overline{\eta}_0(v)$ to $X^+(\R)$ (as functions), it also generates a holonomic $\mathscr{D}_{X^+_{\CC}}$-module by \cite[Theorem 3.2.3]{HTT08}. This completes the proof.
\end{proof}

Now we move to the Archimedean functional equation.

\begin{proof}[Proof of Theorem \ref{prop:convergence-meromorphy-beyond}]
	The strategy is the same as \cite[\S 6]{Li19}. Let us briefly review the ingredients which need modification for non-spherical $X$. Most of them involve differential operators, and the required works are done in \S\S\ref{sec:invariant-theory}--\ref{sec:D}.

	\begin{itemize}
		\item Finiteness of multiplicities, guaranteed by Proposition \ref{prop:finiteness-beyond} in the present context.
		\item A decomposition of $X$, stated in Proposition \ref{prop:rho-decomposition} in the present context.
		\item The regularity of the $\mathscr{D}$-modules generated by $\eta_0(v)$ for all $K$-finite $v \in V_\pi$, where $K$ and $\overline{K}$ are chosen as in Lemma \ref{prop:max-cpt}. The counterpart on $Y^+$ is known by \cite{Li19b}, and this property is preserved under pull-back to $X^+$; see \cite[Theorem 6.1.5]{HTT08}.
		\item Invariant differential operators of Capelli type coming from $(\mathcal{P}_\lambda \otimes \mathcal{Q}_\lambda)^{\overline{G}}$ (where $\lambda \in \mathbf{X}^*_\rho(\check{X})$, see \eqref{eqn:PD-decomp}), or their twists, applied to generalized matrix coefficients of $\pi$. Note that the transition \eqref{eqn:finiteness-beyond} between $X^+$ and $Y^+$ for generalized matrix coefficients is compatible with that for invariant differential operators (see Lemma \ref{prop:diff-op}).
		\item Knop's Harish-Chandra isomorphism $\mathrm{HC}$ for $Y$ and the effect of twists; see Corollary \ref{prop:HC-twist}.
		\item Formula for the top homogeneous component of the image of Capelli operators under $\mathrm{HC}$ (Corollary \ref{prop:top-term}).
	\end{itemize}
	The remaining arguments from \cite{Li19} carry over verbatim.
\end{proof}

\section{Proofs in the non-Archimedean case}\label{sec:pf-nonarch}
Take $F$ to be a non-Archimedean local field of characteristic zero. Let $(G, \rho, X)$ and $N \lhd G$ be as in Hypothesis \ref{hyp:beyond-sphericity}. As before, set $Y^+ := X^+ /\!/ N$ and denote by $\varphi$ the $N$-torsor $X^+ \to Y^+$. We also fix a minimal parabolic subgroup $P_0 = M_0 U_0$ (resp.\ $\overline{P}_0 = \overline{M}_0 \overline{U}_0$) of $G$ (resp.\ $\overline{G}$), chosen such that $P_0 \supset \overline{P}_0$, $M_0 \supset \overline{M}_0$ and $U_0 \supset \overline{U}_0$, which is of course possible.

\begin{proof}[Proof of Theorem \ref{prop:convergence-meromorphy-beyond}]
	The strategy is the same as in \cite[pp.81--82]{LiLNM}. It is based on Igusa's theory (see the formulation before \cite[Corollary 5.1.8]{SV17}) and the smooth asymptotics of scalar-valued generalized matrix coefficients on $Y^+$. We set up a diagram of $G$-varieties
	\begin{equation}\label{eqn:toroidal-embeddings}\begin{tikzcd}
		& \widehat{X} \arrow[d, "p"] \arrow[hookrightarrow, r] & \overline{X} \arrow[dd] \\
		X^+ \arrow[hookrightarrow, ru] \arrow[hookrightarrow, r] \arrow[d, "\varphi"'] & X & \\
		Y^+ \arrow[hookrightarrow, rr] & & \overline{Y}
	\end{tikzcd}\end{equation}
	where
	\begin{compactitem}
		\item all $\hookrightarrow$ are open immersions;
		\item all vertical arrows are dominant;
		\item $Y^+ \hookrightarrow \overline{Y}$ is a smooth toroidal compactification of $\overline{G}$-varieties;
		\item $X^+ \hookrightarrow \widehat{X} \hookrightarrow \overline{X}$ are toroidal embedding of $G$-varieties, and their composition $X^+ \hookrightarrow \overline{X}$ is a smooth toroidal compactification;
		\item $p: \widehat{X} \to X$ is proper birational;
		\item $\overline{X} \to \overline{Y}$ is a $G$-equivariant morphism between toroidal compactifications with respect to $\varphi$, in the sense of \cite[\S 7]{KK16}.
	\end{compactitem}

	Note that in the smooth toroidal compactifications $X^+ \hookrightarrow \overline{X}$ and $Y^+ \hookrightarrow \overline{Y}$, the complements of open $G$-orbits are divisors with normal crossings; see \cite[p.31]{SV17} for explanations in the spherical case.

	Here we are using the more general theory in \cite{KK16} of toroidal embeddings. Specifically, via the choice of minimal parabolic and Levi subgroups, one can define the $\Q$-vector spaces $\mathcal{N}_F(X^+)$ (resp.\ $\mathcal{N}_F(Y^+)$) as in \cite[\S 4]{KK16} for $X^+$ (resp. $Y^+$); the toroidal embeddings are determined by fans in these spaces.\footnote{Toric varieties become the simplest instances of toroidal embeddings by taking the reductive group $G$ to be a torus.} The quotient morphism $\varphi: X^+ \to Y^+$ being dominant, it induces $\varphi_*: \mathcal{N}_F(X^+) \to \mathcal{N}_F(Y^+)$. The objects in \eqref{eqn:toroidal-embeddings} are constructed as follows.
	\begin{itemize}
		\item $Y^+ \hookrightarrow \overline{Y}$ arises from a complete fan $\mathcal{F}'$ in $\mathcal{N}_F(Y^+)$, sufficiently fine to ensure the smoothness of $Y$;
		\item similarly, $X^+ \hookrightarrow \overline{X}$ arises from a complete fan $\mathcal{F}$ in $\mathcal{N}_F(X^+)$, and upon refining $\mathcal{F}$ we have a commutative diagram
		\[\begin{tikzcd}
			X^+ \arrow[hookrightarrow, r] \arrow[d, "\varphi"'] & \overline{X} \arrow[d] \\
			Y^+ \arrow[hookrightarrow, r] & \overline{Y} ,
		\end{tikzcd}\]
		see the general theory in \textit{loc.\ cit.};
		\item we take a fan $\mathcal{F}^\flat$ in $\mathcal{N}_F(X^+)$, sufficiently fine so that it defines a smooth toroidal embedding $X^+ \hookrightarrow \widehat{X}$ dominating $X^+ \hookrightarrow X$, by applying \cite[7.7 Proposition]{KK16} to the equivariant dominant morphism $X^+ \hookrightarrow X$;
		\item by taking $\mathcal{F}^\flat$ to have the correct support, the properness of $p: \widehat{X} \to X$ is ensured by the valuative criterion and \textit{loc.\ cit.} (more precisely, see \cite[Theorem 12.13]{Ti11});
		\item upon further refinements, we may assume $\mathcal{F}^\flat$ is a subfan of $\mathcal{F}$, so \cite[7.8 Corollary]{KK16} gives the open immersion $\hat{X} \hookrightarrow \overline{X}$.
	\end{itemize}

	Next, fix $\eta \in \mathcal{N}_\pi(X^+)$, $v \in V_\pi$ and $\xi \in \Schw(X)$ to form $Z_\lambda(\eta, v, \xi)$. Take the corresponding $\lambda_0$, $\eta_0$ and $\xi_0$ as in Remark \ref{rem:eta-0}, i.e.\ we trivialize the bundle of half-densities. To put things into Igusa's setting, following \cite[p.82]{LiLNM}, we consider
	\begin{itemize}
		\item the divisor $D := \widehat{X} \smallsetminus X^+$ in $\widehat{X}$,
		\item the proper morphism $p: \widehat{X} \to X$,
		\item $\theta := p^* (\eta_0(v)) \in C^\infty(X^+; \CC)$ and $\Xi := p^* \xi_0 \in C^\infty_c(\widehat{X}(F); \CC)$.
	\end{itemize}
	Then Remark \ref{rem:eta-0} and a change of variables give
	\begin{equation}\label{eqn:Igusa-integral}
		Z_\lambda(\eta, v, \xi) = \int_{\widehat{X}(F)} \theta |p^* f|^{\lambda - \lambda_0} \Xi \cdot |p^* \Omega| , \quad \lambda \in \Lambda_{\CC}.
	\end{equation}

	We claim that $\theta$ is $D$-finite in the sense of \cite[p.85]{SV17}. By \textit{loc.\ cit.}, the finiteness is preserved by pull-back and it suffices to show that $\theta$ is $\tilde{D}$-finite where $\tilde{D} := \overline{X} \smallsetminus X^+$. Since $\theta$ is pulled-back from $Y^+$, by \eqref{eqn:toroidal-embeddings} and the discussions in \textit{loc.\ cit.}, we are reduced to show that $\eta_0$ is $(\overline{Y} \smallsetminus Y^+)$-finite as a function on $Y^+(F)$. This is guaranteed by the theory of smooth asymptotics in \cite[Corollary 5.1.8]{SV17}.
	
	Assuming that \eqref{eqn:Igusa-integral} converges for $\Re(\lambda) \relgg{X} 0$, the rational continuation then follows from the $D$-finiteness of $\theta$, as in \cite[p.82]{LiLNM}.
	
	It remains to address the convergence. Given \eqref{eqn:toroidal-embeddings} and the smooth asymptotics for $\eta_0$ (or its $\tilde{D}$-finiteness), it remains to repeat the routine arguments of \cite[\S 5.3]{LiLNM}. Note that the cited proof is based on the Cartan decomposition for $Y^+(F)$, which is in turn a consequence of the theory of toroidal embeddings.
\end{proof}

\begin{remark}
	In the proof above, we invoked the results from \cite{SV17} such as the smooth asymptotics and Cartan decomposition, in which $G$ is usually assumed to be split. However, in Hypotheses \ref{hyp:sphericity} and \ref{hyp:beyond-sphericity}, the case of non-split $G$ and symmetric $Y^+$ is also allowed. There are two workarounds.
	\begin{enumerate}
		\item One could admit that the relevant parts of \cite{SV17} extend to non-split $G$, in view of the theory of \cite{KK16}. In particular, \cite[\S 13]{KK16} provides general Cartan decompositions.
		\item One can also use the earlier works \cite{La08, KT08} on $p$-adic symmetric spaces.
	\end{enumerate}

	With either approach, one has to take the fan $\mathcal{F}'$ in $\mathcal{N}_F(Y^+)$ to be sufficiently fine, so that the toroidal compactification $Y^+ \hookrightarrow \overline{Y}$ in \eqref{eqn:toroidal-embeddings} captures all the ``directions to the infinity'' describing the asymptotic behavior of $\eta_0(v)$.
	
	Note that by granting that the results of \cite{SV17} generalizes to non-split groups, the Hypothesis \ref{hyp:sphericity} can be weakened accordingly.
\end{remark}

% Below: for Biblatex...
%\printbibliography

% Below: for bibtex - preferred on arXiv
\bibliographystyle{abbrv}
\bibliography{VariationSato}

%\vspace{1em}
%\begin{flushleft}
%	Wen-Wei \textsc{Li} \\
%	E-mail address: \href{mailto:wwli@bicmr.pku.edu.cn}{\texttt{wwli@bicmr.pku.edu.cn}} \\
%	School of Mathematical Sciences / Beijing International Center of Mathematical Research, Peking University \\
%	No.\ 5 Yiheyuan Road, Beijing 100871, People's Republic of China.
%\end{flushleft}

\end{document}